\documentclass[a4paper,10pt]{paper}

\PassOptionsToPackage{x11names}{xcolor}
\usepackage{graphicx}
\usepackage{amsmath,amsthm}
\usepackage{amssymb}
\usepackage{pgf}
\usepackage{multicol}
\usepackage{mathtools}
\usepackage{titlesec}
\usepackage{setspace}
\usepackage{changes}
\usepackage{comment}

\usepackage{verbatim}
\usepackage[utf8x]{inputenc} % do subequations
\usepackage{tikz}% do grafów
\usepackage{tikz-network}%do strzałek na krawędziach

\usetikzlibrary{shapes.geometric}
\usetikzlibrary{calc}
\newcommand{\be}{\begin{eqnarray}}
\newcommand{\ee}{\end{eqnarray}}
\newcommand{\bd}{\begin{eqnarray*}}
\newcommand{\ed}{\end{eqnarray*}}

\newcommand{\RR}{\mathbb{R}}

\newtheorem{thm}{\small\bf Theorem}
\newtheorem{lem}{\small\bf Lemma}
\newtheorem{rem}{\small\bf Remark}

\newtheorem{df}{\small\bf Definition}
\newtheorem{exam}{\small\bf Example}
\newtheorem{prop}{\small\bf Proposition}

\colorlet{cpiotr}{RoyalBlue3}
\colorlet{cola}{Turquoise3}

\definechangesauthor[name={Piotr}, color=cpiotr]{Piotr}
\definechangesauthor[name={Ola}, color=cola]{Ola}

\newcommand{\padd}[2][]{\added[id=Piotr,remark=#1]{#2}}

\newcommand{\oadd}[2][]{\added[id=Ola,remark=#1]{#2}}

\newcommand{\ola}[1]{\todo[inline,color=cola]{\color{black}#1}}

\title{Burgers' equation revisited: extension of mono-dimensional case on a network}
\author{Piotr Bogusław Mucha$^*$  \& Aleksandra Puchalska\thanks{Institute of Applied Mathematics and Mechanics, University of Warsaw, Banacha 2, 02-097 Warsaw, Poland}}
\date{}

\begin{document}

\maketitle

\centerline{{\em In memory of our friend Anton\'{i}n}}

\begin{abstract}
    The paper deals with the analysis of Burgers' equation on acyclic metric graphs. The main goal is to establish the existence of weak solutions in the $TV$ -- class of regularity. A key point is transmission conditions in vertices obeying the Kirchhoff law. First, we consider positive solutions at arbitrary acyclic networks and highlight two kinds of vertices, describing two mechanisms of flow splitting at the vertex. Next we design rules at vertices for solutions of arbitrary sign for any subgraph of hexagonal grid, which leads to a construction of general solutions with $TV$ -- regularity for this class of networks. Introduced transmission conditions are motivated by the change of the energy estimation. 
\end{abstract}

{\bf Keywords:} Burgers' equation, networks, hexagonal grid, transmission conditions in vertices, conservation laws, weak solutions, PDEs on metric graphs
%Multi-component fluid, regular solutions, maximal regularity, Lorentz spaces.

%\maketitle
%\newpage
\tableofcontents

%{\large Burgers' equation revisited: extension of mono-dimensional case on a network}
%\centerline{PBM and AP}

\newpage

\section{Motivation}\label{sec:motiv}

The subject of this paper is the mono-dimensional, inviscid Burgers' equation which is the simplest model that begins the whole universe of systems of fluid dynamics. From the mechanical viewpoint it is pure transport of the velocity, modelling a creation of water waves. In the language of material derivative it reads $\frac{D}{Dt}u=0$, and by reformulating we obtain the well known equation
\begin{equation}\label{eq:BE}
    \partial_t u+u \partial_x u=0 \mbox{ \ \ on  \ } D \times [0,T).
\end{equation}
The theory says that starting from any smooth, compactly supported initial configuration, the solution does not have to be smooth but it can suffer a jump discontinuity. Thus, waves can create a shock, also called the gradient catastrophe. When we move to the weak formulation non-unique solutions are allowed, and either non-physical shocks or physical rarefaction waves appear. In order to keep the mathematical well-posedness the concept of entropy solutions has been introduced, and then both uniqueness property and decrease in time to zero is guaranteed. The shocks are governed by the Rankine-Hugoniot condition determining the speed of the jump %in terms of the jump of the flux the solution itself
and the Lax condition choosing between continuous and discontinuous solution. The general rule is that the bigger wave overtakes the smaller one, so consequently, for a long time we are not able to say anything about the smaller wave. 
%The only information about it remains coded in the speed of the bigger wave.

In this paper we address the following question: Is there any approach to the Burgers' equation (\ref{eq:BE}) which admits certain preservation of the smaller wave after collision with the bigger one?
The answer is positive, but we are required to take $D$ as a graph.

\section{Problem formulation}\label{sec:BE}

The problem of inviscid Burgers equation on network does not appear itself in the literature, but its viscid counterpart was addressed both in theoretical considerations in \cite{Bressloff1997,Hinz2020} and in numerical simulations in \cite{B_numer}. Nevertheless, the problem fits to the general framework of conservation laws on networks that has been developed for about thirty years and still receives considerable interest  \cite{GarPic2009, rev_2014, Musch2022}. 

The majority of papers in this topic are related to traffic modelling, see for instance \cite{CGP, Holden2018, Goatin2020}, being the continuation of considerations on Lighthill-Whitham model \cite{Light-Whit1955}. The major difference of the mentioned approach compared to this paper is the appearance of non-convex flux which enforces the application of either wave-front tracking approximations or vanishing viscosity methods \cite{CocGar2010}. Considering pure Burgers' equation allows us to use methodology known from Hamilton-Jacobi equation \cite[Sec.~3.3]{Evans} and consequently to obtain an explicit solution being a counterpart of the well-known Lax-Oleinik formula.

Although all above papers refer to the domain of examination as the general network, their investigations concentrate on so-called star graphs. It is not the limitation in the sense of the existence theorems since preservation of finite velocity at finite networks does not allow for a blow up. Nevertheless, the global perspective related to network properties is missing. %In these considerations we concentrate on the graph as a whole and what is more, we do not limit ourselves to the flow of a fixed direction.

The main subject of this study is the development of a coherent language of description of fluid-type equations on metric graphs, that allows us to pursue from Burgers' equation to multi-dimensional systems like Navier-Stokes or compressible Euler. Incorporation of a full network structure into the model suggests the interpretation of a problem either as extension of mono-dimensional case or as non-standard discretization of a state space. A new structural approach allows us to address strictly mathematical questions like: the existence of new kinds of blows up or the loss of control of regularity of solutions. Successful development of the above theory gives  hope for alternative techniques for proving blowup and uniqueness criteria for the classical fluid systems. 
To this end, we begin in this paper with addressing two preliminary questions.

Firstly, what is the appropriate description of the flow in vertices? Since Burgers' equation is very close to both the Euler system \cite{FeNo} and the continuity equations \cite{NoPo}, the natural approach is to look at the change of the energy at redistribution points, namely taking the maximal or minimal change of the energy at vertices. The second question is the relation between the pure mono-dimensional case and the network counterpart.

%The latest results in this field suggest addressing the following questions.  

%We want to make the first step in analysis of fluid-type equations on metric graphs. From our perspective this approach is very promising. The study of the Burger equation delivers a few important equations. One of them is the behaviour, or rather the definition, of the flow in vertices. Since we are very close to the Euler system \cite{FeNo} and the continuity equations \cite{NoPo} we shall look at the change of the energy at redistribution points.  A natural choice is taking the maximal or minimal change of the energy at vertices. The second question is the difference from the pure mono-dimensional case. Here we think about strictly mathematical questions like new kinds of blows up or loss of control of regularity of solutions. We want to prepare the area to deal with multi-dimensional systems like Navier-Stokes or compressible Euler.

%The development of our approach on graphs shall lead to an alternative technique for proofing blowup and uniqueness criteria for the classical fluid systems.

\subsection{Graph theory toolbox}\label{sec:graph}
Consider $G=(V,E,\mathcal{L},\Phi)$ a directed, weighted and finite tree with no multiple edges. Namely, let 
$$
V:=\left\{v_i:\,i\in I\right\}, \mbox{ \ \ for \ \ } I=\left\{1,\ldots,n\right\},
\mbox{ \ \ and \ \ } E:=\left\{e_j:\,j\in J\right\}, \mbox{ \ \ for \ \ } J=\left\{1,\ldots,m\right\},
$$
be respectively \emph{sets of vertices and edges of a graph}; while $\mathcal{L}:E\rightarrow \mathbb{R}_+$ be a \emph{weight (length) function} of the edge; $e_j\mapsto l_j$ for any $j\in J$. 

The structure of the network is defined by \emph{incidence matrix} $\Phi\in M_{n\times m}(\mathbb{R})$, $\Phi=(\phi_{ij})_{i\in I,j\in J}=\Phi^+-\Phi^-$ such that $\Phi^+=(\phi^+_{ij})_{i\in I,j\in J}$ and $\Phi^-=(\phi^-_{ij})_{i\in I,j\in J}$ satisfy conditions
	\begin{equation*}
		\phi^{+}_{ij}=\left\{\begin{array}{ll}1 &\text{if}\quad \stackrel{e_j}{\rightarrow}v_i\\
0 &\text{otherwise},\end{array}\right.
\qquad \phi^{-}_{ij}=\left\{\begin{array}{ll}1 &\text{if}\quad v_i\stackrel{e_j}{\rightarrow}\\
0 &\text{otherwise}.\end{array}\right.
		\end{equation*}
If $\phi_{ij}\neq 0$, we say that edge $e_j$ is \emph{incident} to $v_i$. We say that there exists a multiple edge between vertices $v_i, v_k\in V$ if there exist two edges $e_p, e_q\in E$ such that, for $z=p,q$, $\phi^+_{kz}=1$ and $\phi^-_{iz}=1$. Hence, the lack of multiple edges provides a uniqueness of such assignment $e_j=(v_i,v_k)\in E$ for some $v_i,v_k\in V$. In further consideration we call $v_i$ a \textsl{head} and $v_k$ a \textsl{tail }of the edge $e_j$. The vertex $v_i$ is a \emph{source} or a \emph{sink} if respectively $\phi^+_{ij}=0$ or $\phi^-_{ij}=0$ for any $j\in J$. 
 
 By the \emph{path} in the graph we understand a finite sequence of edges $p_i=e_{k_1},\ldots,e_{k_{N_i}}$ such that for any $e_{k_j},e_{k_{j+1}}$ there exists a vertex $v_{k_j}\in V$ such that 
 $$\stackrel{e_{k_j}}{\rightarrow} v_{k_j} \stackrel{e_{k_{j+1}}}{\rightarrow}\quad  \text{(equivalently}\,\, \phi_{k_jk_j}^+=1=\phi^-_{k_jk_{j+1}}\text{)}$$
 for $j=1,...,l-1$. It means the path is of the following form
 $$
 v_{k_0} \stackrel{e_{k_1}}{\rightarrow} v_{k_1} \stackrel{e_{k_2}}{\rightarrow} v_{k_2} \stackrel{e_{k_3}}{\rightarrow} 
 ...  \stackrel{e_{k_{N_i-1}}}{\rightarrow} v_{k_{N_i-1}} \stackrel{e_{k_{N_i}}}{\rightarrow} v_{k_{N_i}}. 
 $$
 By the \emph{length} $N_i$ of a path $p_i$ we understand the number of edges on the path, while by \emph{weighted length} $L_i$ the weights' sum of all edges on the path.
 
 We say that a graph is \emph{connected} if there exists at least one path between every two vertices. A closed path, namely $v_{k_0}=v_{k_{N_i}}$, is a \emph{cycle} and the graph is called \emph{acyclic} if it has no cycles. Finally, we say that a graph is a \emph{directed tree} if it is connected and has no cycles. 
 
 In the following considerations we refer to the special examples of trees being a restriction of finite graphs. We say that $G'=(V',E', \mathcal{L}',\Phi')$ is a subgraph of a graph $G=(V,E,\mathcal{L},\Phi)$ if it satisfies the conditions
 \begin{equation*}
 V'\subseteq V,\quad E'=E|_{V'\times V'}, \quad \mathcal{L}|_{E'},\quad \Phi'=\Phi|_{I'\times J'},
 \end{equation*}
 where $I'=\left\{i\in I:\,\, v_i\in V'\right\}$ and $J'=\left\{j\in J:\,\,e_j\in E'\right\}$.
 \begin{df}\label{def:v-subgraph}
 Consider $G=(V,E,\mathcal{L},\Phi)$ and $v_i\in V$. We say that $G_i=(V_i,E_i,\mathcal{L}_i,\Phi_i)$ is a \emph{$v_i$-subgraph} of $G$ if
 \begin{equation}\label{eq:G_i}
     V_i:=\left\{v_j\in V:\,\,j\in J_i\right\},\qquad\text{and}\quad J_i:={\left\{j\in I:\,\, \phi_{ij}\neq 0\right\}}. 
 \end{equation}
 \end{df}

\begin{df}\label{def:p+h}
A \emph{path graph $P_m$} is any connected subgraph of 1D Cartesian grid $P=(V_P,E_P,\mathcal{L}_P,\Phi_P)$   
\begin{equation*}
    V_P=\left\{v_i:\,i\in \mathbb{Z}\right\},\quad E_P=\left\{e_j:\,i\in \mathbb{Z}\right\}, \quad \mathcal{L}_P\equiv 1,\quad \Phi_P=(\phi_{ij})_{i,j\in \mathbb{Z}},\,\,\phi_{ij}=\left\{\begin{array}{ll}
         \phantom{x}1& \text{for}\,\,i=j \\
         -1&\text{for}\,\,i=j-1\\
         \phantom{x}0&\text{otherwise}
    \end{array}\right.
    \end{equation*}
having \underline{ $m$ edges}. \newline
\indent By the \emph{honeycomb tree $H_m$} we understand any connected subgraph of directed hexagonal lattice $H=(V_H,E_H,\mathcal{L}_H)$
\begin{eqnarray*}
    &V_H=\left\{v_{(p+q,-q,p)},v_{(p+q+1,-q,p)}:\, p,q\in \mathbb{Z}\right\},\qquad \mathcal{L}_H=1,\\ &E_H=\left\{(v_{(p+q,-q,p)},v_{(p+q+1,-q,p)}),(v_{(p+q,-q,p)},v_{(p+q,-q+1,p)}), (v_{(p+q+1,-q,p)},v_{(p+q+1,-q,p+1)}):\, p,q\in \mathbb{Z}\right\}
\end{eqnarray*}
having \underline{$m$ edges}.  In further considerations we refer to $v_{(p+q,-q,p)}$ as vertex of the \emph{first kind} while to $v_{(p+q+1,-q,p)}$ as the vertex of the \emph{second kind}, see Fig. \ref{fig:T1&2}.
\end{df} 
 Note that any vertex of hexagonal lattice $H$ is described by a triple of type $(p+q,-q,p)$ with two parameters $p,q$, which corresponds to the three directions on the honeycomb.
 
 Define now the \emph{in- and out degree of vertex} $v_i$ which is the number of edges having respectively a tail or a head in vertex $v_i$, namely
 \begin{equation*}
     \text{deg}_{+}(v_i)=\sum_{j\in J}\phi^+_{ij},\qquad \text{deg}_{-}(v_i)=\sum_{j\in J}\phi^-_{ij}, \qquad\text{and}\qquad \text{deg}(v_i)=\text{deg}_{+}(v_i)+\text{deg}_{-}(v_i).
 \end{equation*}
 Then using the notation from Def. \ref{def:p+h}, we have for $p,q\in \mathbb{Z}$
 \begin{eqnarray*}
 &\text{deg}_{+}(v_{(p+q,-q,p)})=1,\qquad  \text{deg}_{-}(v_{(p+q,-q,p)})=2;\\
 &\text{deg}_{+}(v_{(p+q+1,-q,p)})=2,\qquad  \text{deg}_{-}(v_{(p+q+1,-q,p)})=1.
 \end{eqnarray*}
 Considering the restriction of $H$ to the subgraph we obtain also additional types of vertices $v$ being sources ($\text{deg}_{+}(v)=0$, $\text{deg}_{-}(v)\in\left\{1,2\right\}$), sinks ($\text{deg}_{+}(v)\in\left\{1,2\right\}$, $\text{deg}_{-}(v)=0$) or vertices of the path graph ($\text{deg}_{-}(v)=\text{deg}_{+}(v)=1$).

 Furthermore, we introduce a \emph{direction of a vertex} $v_i\in V$ as an ordered pair of sets $D_i=({D}^{in}_i,{D}^{out}_i)$, $D^{in}_i,D^{out}_i\subset E$ such that
\be\label{eq:dir}
D^{in}_i:=\left\{e_j\in E:\,\, \stackrel{e_j}{\rightarrow}v_i\right\}, \qquad \text{and}\qquad  D^{out}_i:=\left\{e_j\in E:\,\, v_i \stackrel{e_j}{\rightarrow}\right\}.
\ee
Since directed trees $G$ do not have loops, therefore $D_i^{in}\cap D_i^{out}=\emptyset$, for any $i\in I$. If we change the vertex $v_i$ into $v_i'$ in the way that the parameterization of all edges incident to the vertex $v_i$ become opposite, we say that $v_i$ and $v_i'$ have \emph{the opposite direction}. In the case of honeycomb trees the direction of verities of the first and second kind are, for $p,q\in \mathbb{Z}$, the following
\begin{eqnarray*}
    &D_{(p+q,-q,p)}=\left(D^{in}_{(p+q,-q,p)},D^{out}_{(p+q,-q,p)}\right)\qquad D^{in}_{(p+q,-q,p)}=\left\{(v_{(p+q,-q,p-1)},v_{(p+q,-q,p)})\right\},\\
    &D^{out}_{(p+q,-q,p)}= \left\{(v_{(p+q,-q,p)},v_{(p+q+1,-q,p)}), (v_{(p+q,-q,p)},v_{(p+q,-q+1,p)})\right\},\\
    &D_{(p+q+1,-q,p)}=\left(D^{in}_{(p+q+1,-q,p)},D^{out}_{(p+q+1,-q,p)}\right)\qquad D^{out}_{(p+q+1,-q,p)}=\left\{(v_{(p+q+1,-q,p)},v_{(p+q+1,-q,p+1)})\right\}\\
    &D^{in}_{(p+q+1,-q,p)}=\left\{(v_{(p+q,-q-1,p)},v_{(p+q+1,-q,p)}), (v_{(p+q-1,-q,p)},v_{(p+q+1,-q,p)})\right\}.
\end{eqnarray*}
The above distinction is crucial to the considerations in Section \ref{sec:gen_trans}.
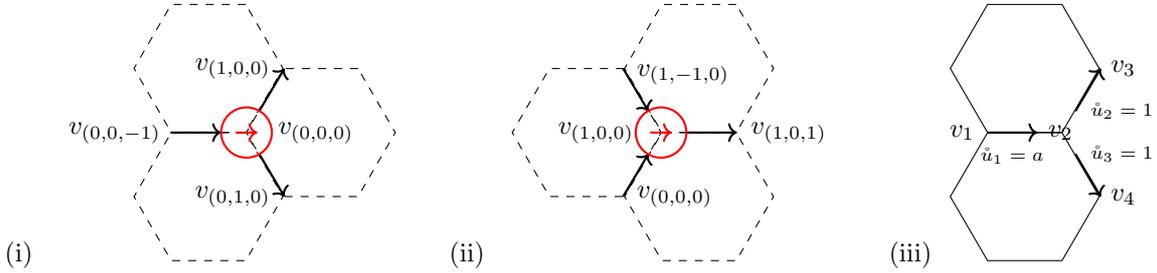
\begin{figure}[h]
\center
(i)\quad \begin{tikzpicture}
\def\tmp{1.7};
\foreach \x in {0,1}{\path[draw=gray,dashed] (0,\tmp*\x)
node[
regular polygon,%fill=gray!50,%wypelłnienie kolorem
draw=none,
regular polygon sides=6,
draw,
inner sep=.6cm,
] (hexagon) {};}	
\path[draw=gray, dashed] (0.985*3/2,\tmp/2)
node[
regular polygon,%fill=gray!50,%wypelłnienie kolorem
draw=none,
regular polygon sides=6,
draw,
inner sep=.6cm,
] (hexagon) {};	
	
\draw[->,thick] (-1/2,\tmp/2) -- node[below=2pt] {} (1/2-1/3,\tmp/2);
\draw[->,thick] (1/2+1/6,\tmp/2+\tmp/6) -- node[left=2pt] {} (1,\tmp);
\draw[->,thick] (1/2+1/6,\tmp/2-\tmp/6) -- node[right=1pt] {} (1,0);
\draw (1/2,\tmp/2) node[right]{\,\,\,\,\,$v_{(0,0,0)}$};	
\draw (-1/2,\tmp/2) node[left]{$v_{(0,0,-1)}$};	
\draw (1,\tmp) node[left]{\,$v_{(1,0,0)}$\,};	
\draw (1,0) node[left]{\,$v_{(0,1,0)}$\,};

\draw[red,thick] (1/2,\tmp/2) circle (1/3);
\draw[->,red,thick] (1/2-1/7,\tmp/2) -- node{} (1/2+1/7,\tmp/2);

\end{tikzpicture}\qquad (ii)\quad \begin{tikzpicture}
\def\tmp{1.7};
\foreach \x in {0,1}{\path[draw=gray,dashed] (0,\tmp*\x)
node[
regular polygon,%fill=gray!50,%wypelłnienie kolorem
draw=none,
regular polygon sides=6,
draw,
inner sep=.6cm,
] (hexagon) {};}	
\path[draw=gray, dashed] (-0.985*3/2,\tmp/2) %\tmp/2
node[
regular polygon,%fill=gray!50,%wypelłnienie kolorem
draw=none,
regular polygon sides=6,
draw,
inner sep=.6cm,
] (hexagon) {};
\draw[->,thick] (-1/6,\tmp/2) -- node[below=2pt] {} (1/2,\tmp/2);
\draw[->,thick] (-1,\tmp) -- node[left=2pt] {} (-2/3,2*\tmp/3);
\draw[->,thick] (-1,0) -- node[left=1pt] {} (-2/3,\tmp/3);
\draw (-1/2,\tmp/2) node[left]{$\tiny{v_{(1,0,0)}}$\,\,\,\,};			\draw (1/2,\tmp/2) node[right]{\,$\tiny{v_{(1,0,1)}}$};
\draw (-1,0) node[right]{\,$\tiny{v_{(0,0,0)}}$};
\draw[red,thick] (-1/2,\tmp/2) circle (1/3);
\draw (-1,\tmp/1.05) node[right]{\,$\tiny{v_{(1,-1,0)}}$};
\draw[->,red,thick] (-1/2-1/7,\tmp/2) -- node[left=1pt] {} (-1/2+1/7,\tmp/2);
\end{tikzpicture}\qquad (iii) \begin{tikzpicture}
\def\tmp{1.7};
\foreach \x in {0,1}{\path[draw=white] (0,\tmp*\x)
node[
regular polygon,
draw=none,
regular polygon sides=6,
draw,
inner sep=.6cm,
] (hexagon) {};}	
%\draw[->,thick] (-1/2,\tmp/2) -- node[below=2pt] {$e_1$} (1/2,\tmp/2);
%\draw[->,thick] (1/2,\tmp/2) -- node[left=2pt] {$e_2$} (1,\tmp);
%\draw[->,thick] (1/2,\tmp/2) -- node[right=1pt] {$e_3$} (1,0);
\draw[->,thick] (-1/2,\tmp/2) -- node[below=2pt] {\scriptsize{$\mathring{u}_1=a$}} (1/2-1/3,\tmp/2);
\draw[->,thick] (1/2+1/6,\tmp/2+\tmp/6) node[right=2pt] {\scriptsize{$\mathring{u}_2=1$}} --  (1,\tmp);
\draw[->,thick] (1/2+1/6,\tmp/2-\tmp/6) node[right=2pt] {\scriptsize{$\mathring{u}_3=1$}} --  (1,0);

\draw (-1/2,\tmp/2) node[left=1pt] {$v_1$};	
\draw (1/2-1/3,\tmp/2) node[right=.5pt] {$v_2$};
\draw (1,\tmp) node[above=0.5pt, right=.25pt] {$v_3$};
\draw (1,0) node[below=0.5pt, right=.25pt] {$v_4$};
\end{tikzpicture}

	\caption{Two kinds of vertices in honeycomb trees $H_{15}$. (i) $v_{(0,0,0)}$ is of the first kind  (ii) $v_{(1,0,0)}$ is of the second kind. Vertices' direction is denoted in the symbolic way in red. In (iii) a metric honeycomb tree $\mathcal{H}_3$ introduced in Example \ref{exam:uniq_honey}.}
		\label{fig:T1&2}
\end{figure}

 Finally, let us remind that for any tree it is possible to re-enumerate edges in the way that for any two edges $e_{s}, e_{j}\in E$, and for any chosen path $e_{s}=e_{k_1},\ldots,e_{k_N}=e_{j}$; $k_i<k_{i+1}$ for all $i\in 1,\ldots,N-1$. Additionally in the following considerations we chose the enumeration of edges in the way that all sources are associated with the first few edges, namely sources are heads of the edges $e_{i}$, $i=1,\ldots, s$. We call such numeration \emph{an increasing order of edges} and note that two trees with an increasing order of edges are homomorphic.
 
 \subsection{Introduction of  metric graphs}\label{sec:m_graph}
 
 To introduce a metric space into consideration we associate each edge of a graph with a compact interval in the following way for $d:E\rightarrow \mathcal{B}(\mathbb{R})$ let $d(e_j)=[0,l_j]$. We say that $\mathcal{G}=(G,d)$ is a \emph{directed metric graph}. In what follows we always consider the parametrisation of an edge that agrees with the direction of an edge. By an abuse of notation we shall denote a metric edge $d(e_j)$ simply by $e_j$, the vertices at the endpoints of the edge $e_j=(v_i,v_k)$ by $e_j(0):=v_i$ and $e_j(l_j):=v_k$. Further, when considering a function $f_j$ defined on the metric edge $d(e_j)=[0, l_j]$, we shall occasionally write $f(v_i):=f (s)$ if $e_j(s)=v_i$ for $s=0,l_j$. By the function defined on the metric graph we understand a vector-valued function $f:[0,1]\rightarrow \mathbb{R}^m$ such that $f(x)=(f_j(l_jx))_{j\in J}$, where $f_j:[0,l_j]\rightarrow \mathbb{R}$ is defined on the edge $e_j$. 

\smallskip 

The main idea of this paper is to find the function defined on the metric graph that satisfies both the weak formulation of Burgers' equation on edges and certain transmission conditions in vertices. Based on general knowledge of the mono-dimensional case, it is obvious that the direction of a flow can disagree with the parameterization of an edge. Although it does not cause a difficulty on the edge, it complicates transmission conditions. %We deal with this problem in two steps. Firstly, we consider only non-negative solutions for general trees in Section \ref{sec:posit_NBE}, and then restrict our considerations to honeycomb trees in order to obtain general solutions in Section \ref{sec:gen_NBE}. 
To define well conditions in vertices we extend the classical notion of \emph{weighted adjacency matrix of a line graph} $\mathcal{B}=(b_{ij})_{i,j\in J}$, which in the standard setting reads 
	\begin{equation}\label{eq:adj}
		b_{jk}\neq 0 \quad \text{if}\quad \exists_{v_i}\,\,\stackrel{e_k}{\rightarrow}v_i\stackrel{e_j}{\rightarrow}\qquad\text{and}\qquad b_{jk}=0\quad \text{otherwise}.
		\end{equation}
Consider the following operators $\mathcal{B}^{pq}=(b_{jk}^{pq})_{j,k\in J}$, 
%$Q=(l_j)_{j\in J}$ 
for $p,q\in \{0,1\}$ 
%or ,$l_i $ and $q=0,l_j$
such that
\begin{subequations}\label{eq:adj'}
\begin{align}\label{eq:B01}
		b_{jk}^{01}&\geq 0 \quad \text{if}\quad \exists_{v_i}\,\,\stackrel{e_k}{\rightarrow}v_i\stackrel{e_j}{\rightarrow}\qquad\text{and}\qquad b_{jk}^{01}=0\quad \text{otherwise};\\ \label{eq:B00}
		b_{jk}^{00}&\geq 0 \quad \text{if}\quad \exists_{v_i}\,\,\stackrel{e_k}{\leftarrow}v_i\stackrel{e_j}{\rightarrow}\qquad\text{and}\qquad b_{jk}^{00}=0\quad \text{otherwise};\\ \label{eq:B10}
			b_{jk}^{10}&\geq 0 \quad \text{if}\quad \exists_{v_i}\,\,\stackrel{e_k}{\leftarrow}v_i\stackrel{e_j}{\leftarrow}\qquad\text{and}\qquad b_{jk}^{10}=0\quad \text{otherwise};\\ \label{eq:B11}
				b_{jk}^{11}&\geq 0 \quad \text{if}\quad \exists_{v_i}\,\,\stackrel{e_k}{\rightarrow}v_i\stackrel{e_j}{\leftarrow}\qquad\text{and}\qquad b_{jk}^{11}=0\quad \text{otherwise}.
\end{align}
\end{subequations}
Note that the new approach to adjacency matrix definition given in \eqref{eq:adj'}, unlike the classical one \eqref{eq:adj}, allows for the lack of flow between two edges even though they are physically connected. Obviously $\mathcal{B}^{01}=\left(\mathcal{B}^{10}\right)^T$, but we distinguish those cases due to its different meaning in the sense of flow. Note that if $b_{jk}^{01}\geq 0$, then using the notation from \eqref{eq:B01} and \eqref{eq:dir}, $e_k\in D^{in}_i$ and $e_j\in D^{out}_i$. On the other hand for $b_{jk}^{10}\geq 0$, $e_j\in D^{in}_i$ and $e_k\in D^{out}_i$. Consequently, in the first case the direction of vertex $v_i$ is opposite to the direction of the vertex in the second case. 

If we replace $1$ with arbitrary nonzero coefficients in matrices $\mathcal{B}$, $\mathcal{B}^{pq}$ we arrive at unweighted counterparts of matrices, we call them \emph{adjacency matrices of a line graph}, and denote them by $\mathcal{\overline{B}}$, $\mathcal{\overline{B}}^{pq}$.

Due to the change in the definition of adjacency matrices, it is possible to find a path in the metric graph in which there is no possibility of flow from one edge, say $e_k$, to another $e_j$ due to vanishing of coefficients $b_{jk}^{pq}$, $p,q\in\left\{0,1\right\}$, $j,k\in J$. Therefore in the whole paper we distinguish the definition of path in the graph $G$ and in its metric counterpart $\mathcal{G}$. By the \emph{path} in the metric graph we understand a finite sequence of edges $p_i=e_{k_1},\ldots,e_{k_{N_i}}$ such that for any $e_{k_j},e_{k_{j+1}}$ there exists a pair $(p,q)$, $p,q\in\left\{0,1\right\}$ such that $b_{k_{j+1}k_j}^{pq}\neq 0$. The notions of path length $N_i$ and weighted path length $L_i$ remain unchanged.

\subsection{Burgers' equation on the network}\label{sec:NBE}
Let us defined Burgers' equation on the metric graph $\mathcal{G}$ in the line with motivation, see equation \eqref{eq:BE},
\begin{equation}
    \partial_t u+u \partial_x u=0 \mbox{ \ \ on  \ } \mathcal{G} \times [0,T).
\end{equation}
Namely, let $u=(u_j(l_j\cdot))_{j\in J}$ be the function defined on the metric graph $\mathcal{G}$ which satisfies 
\begin{subequations}\label{eq:Burgers}
\begin{align}
    \partial_t u_j(x,t)+u_j(x,t)\partial_x u_j(x,t)&=0, &x\in [0,l_j],\, t\in[0,T),\label{eq:B_eq}\\
u_j(x,0)&=\mathring{u}_j(x), &x\in[0,l_j],\label{eq:B_init}
\end{align}
\end{subequations}
for every coordinate $j\in J$. Now let us derive the transmission conditions that incorporate the network structure into the formulation from one hand, and allow for the flow that agrees with the physical motivation from another.

Let us start with the formulation of transfers that comes from the generalisation of vertex conditions for network transport, see \cite[Sec. 3a]{KFP}. Consider operators $u \mapsto \mathcal{B}_z(u)\in M_{m}(\mathbb{R})$, $z=0,1$ and for almost all $t\in[0,T)$ assume that 
\begin{equation}\label{eq:GTC}
\mathcal{B}_0(u)u(0,t)+\mathcal{B}_1(u) u(1,t)=0, \qquad \text{with}\quad \mathcal{B}_z=\mathcal{B}^{z0}+\mathcal{B}^{z1},\quad \text{defined in \eqref{eq:adj'}.}
\end{equation}
%\begin{equation}\label{eq:GTC}
%\mathcal{K}_0(t)\delta_0u(\cdot,t)+\mathcal{K}_1(t) \delta_l u(\cdot,t)=0,
%\end{equation}
%where $\delta_0,\delta_l\in \textcolor{red}{??}$ are respectively the evaluation at $x=0$ and 
%\begin{equation*}
%    \delta_l u(\cdot,t):=(u_j(l_j,t))_{j\in J}.
%\end{equation*}
Obviously such a general formulation has to be specified for a number of reasons. Even in the linear case, when $\mathcal{B}_0, \mathcal{B}_1$ are independent of $u$, the uniqueness of the solution to \eqref{eq:GTC} strictly depends on their rank. Furthermore, there is no clear relation with a graph structure because again for arbitrary operators $\mathcal{B}_0,\mathcal{B}_1\in M_{m}(\mathbb{R})$, it is not always possible to build the graph,  not mentioning  the directed tree that is the object of these considerations. For details see \cite{BF2015}. 

Let us draw your attention to one property that is important in further considerations. If the direction of flow disagrees with the parametrization it may allow for a cyclic flow along the edges even though the graph is a directed tree.
\begin{exam}\label{exam:circle}
Consider a graph $G=(V,E,\mathcal{L},\Phi)$ such that 
\begin{equation}\label{eq:circle}
    V=\left\{v_i:\,i=1,2,3\right\},\quad E=\left\{e_j:\,j=1,2,3\right\},\quad L\equiv 2\pi \quad \text{and}\quad \Phi=\left[\begin{array}{ccc}
    -1&0&-1\\1&-1&0\\0&1&1\end{array}\right],
\end{equation}
presented also in Figure \ref{fig:circle}. Problem \eqref{eq:Burgers} -- \eqref{eq:GTC} such that
\begin{equation*}
    \mathcal{B}_0=\left[\begin{array}{ccc}
    1&0&1\\0&-1&0\\0&0&0\end{array}\right],\quad \mathcal{B}_1=\left[\begin{array}{ccc}
    0&0&0\\1&0&0\\0&1&1\end{array}\right],\quad\text{and}\quad \mathring{u}_1,\mathring{u}_2>0,\quad \mathring{u}_3<0.
\end{equation*}
is equivalent locally in time with the Burgers' equation on the circle with radius $r=3$.
\end{exam}

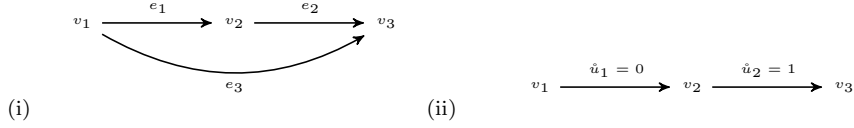
\begin{figure}[h]
\begin{center}
\begin{tikzpicture}[->,>=stealth',shorten >=1pt,auto,node distance=1.5cm,semithick]
\tikzstyle{every state}=[fill=none, draw=none, text=black]

\node (A) at (0,0) {\tiny{$v_1$}};
\node (B) at (2,0) {\tiny{$v_2$}};
\node (C) at (4,0) {\tiny{$v_3$}};

\path (A) edge node [above] {\tiny{$e_1$}} (B);
\path (B) edge node  [above] {\tiny{$e_2$}} (C);
\path (A) edge [bend right] node [below] {\tiny{$e_3$}} (C);
\end{tikzpicture} \qquad \qquad
\begin{tikzpicture}[->,>=stealth',shorten >=1pt,auto,node distance=1.5cm,semithick]
\tikzstyle{every state}=[fill=none, draw=none, text=black]

\node (A) at (0,0) {\tiny{$v_1$}};
\node (B) at (2,0) {\tiny{$v_2$}};
\node (C) at (4,0) {\tiny{$v_3$}};

\path (A) edge node [above] {\tiny{$\mathring{u}_1=0$}} (B);
\path (B) edge node  [above] {\tiny{$\mathring{u}_2=1$}} (C);
\end{tikzpicture} \end{center} \vspace{-.5cm}
\hspace{2cm}\footnotesize{(i)} \hspace{5cm} \footnotesize{(ii)}
\caption{Illustration depicts two network structures: the first is a graph $G$ defined in \eqref{eq:circle} and considered in Example \ref{exam:circle}; while the second a metric path graph $\mathcal{P}_2$ introduced in Example \ref{exam:uniq_path}.}
		\label{fig:circle}
\end{figure}
In the Example \ref{exam:circle} the cyclic structure appeared due to the disturbance of a flow in vertices $v_1$ and $v_3$. Note that the direction of the vertex $v_1$ is $D_1=(D_1^{in},D^{out}_1)=\left(\emptyset,\left\{e_1,e_2\right\}\right)$ while the mass flows from the edge $e_3$ into $e_1$. Similar problem appeared in $v_3$. In the further considerations we allow for the flow to go in line with the vertex direction and in the opposite direction. We assure, however, that there is no exchange of mass between edges in the sets $D_i^{in}$ (as well $D_i^{out}$) for any $i\in I$, namely
\begin{equation}\label{eq:n_return}
    \mathcal{B}^{00}=\mathcal{B}^{11}=0.
\end{equation}
Let us now fix the vertex $i\in I$, the moment $t\in[0,T)$ and we consider two cases. If the flow at $t$ agrees with the direction of a vertex. Then the transmission conditions in vertex $v_i$ read
\begin{equation}\label{eq:trans+}
    u_j(0,t)=\sum_{\left\{s\in J:\,e_s\in D_i^{in}\right\}}\,b^{01}_{js}(u)u_s(1,t),\qquad \text{for}\,\,j\in J\,\,\text{such that}\,\,\phi_{ij}^-\neq 0, 
\end{equation}
where $(b^{01}_{js})_{j,s\in J}$ is the adjacency matrix defined in \eqref{eq:B01}. Similarly, for the flow opposite to the vertex direction we have
\begin{equation}\label{eq:trans-}
    u_j(1,t)=\sum_{\left\{s\in J:\,e_s\in D_i^{out}\right\}}\,b^{10}_{js}(u)u_s(0,t),\qquad \text{for}\,\,j\in J\,\,\text{such that}\,\,\phi_{ij}^+\neq 0,
\end{equation}
with $(b^{10}_{js})_{j,s\in J}$ is the adjacency matrix defined in \eqref{eq:B10}. In particular we note that for considered problem matrices $\mathcal{B}^{00}$ and $\mathcal{B}^{11}$ defined respectively in \eqref{eq:B00} and \eqref{eq:B11}, vanish. %In the meantime there is a unique choice of transmission conditions and the graph structure is preserved.

\begin{df}\label{def:conserv}
We say that system \eqref{eq:Burgers}--\eqref{eq:trans+}--\eqref{eq:trans-} is the strong formulation of Burgers' equation on the metric tree $\mathcal{G}$.
\end{df}
The above definition is formal, still the relation $\mathcal{B}(u)$ is not given.
In order to move from strong to weak formulation we introduce a set of smooth functions over $\mathcal{G}$. Namely, the functions smooth over the edges which agree on germs given in each  vertex $v_i$; with the neighbourhood oriented in line with direction $D_i$. Below we give a weaker definition, which always allows  determining  the differentiation by parts.

%\ola{Chcemy to pisac w formie defnicji? Jak dla mnie niekoniecznie} 

\begin{df}\label{def:smooth}
We say that $f=(f_j(l_j\cdot))_{j\in J}$ defined on the metric graph $\mathcal{G}$ is smooth on $\mathcal{G}$, and we write $f\in C^\infty(\mathcal{G})$, if the following conditions hold
\begin{enumerate}
    \item[(i)]  $f_j(l_j\,\cdot)\in C^\infty[0,l_j]$ \hspace{.4cm}for any $e_j \in E$;
    \item[(ii)] for any $v_i\in V$, and any $k \in \mathbb{N}$
$$
\partial^{(k)}f_{j}(l_j \,\cdot)=\partial^{(k)}f_{k}(0) 
\mbox{ \ \ for all \ \ } e_j \in D^{in}_i \mbox{ \ and \ } 
e_k \in D^{out}_{i}.
$$
\end{enumerate}

%(i) \ \ $f|_{e_s}=f_s\in C^\infty[0,l_s]$ for $e_s \in E$;

%(ii) \ \ for any vortex $v_i$ from $V$, and any $k \in \mathbb{N}$
%$$
%\partial^{(k)}f_{e_j}(d_j)=\partial^{(k)}f_{e_k}(0) 
%\mbox{ \ \ for all \ \ } e_j \in E^{in}_i \mbox{ \ and \ } 
%e_k \in E^{out}_{i}.
\end{df}
Consider now a function $\phi:[0,1]\times [0,\infty)\rightarrow \mathbb{R}^m$ %with a compact support such that for any $t\in[0,\infty)$
, $\phi(\cdot,t)\in C^\infty(\mathcal{G})$. In what follows the product of two vector functions is understood in the sense of the Hadamard product, namely $fg=(f_jg_j)_{j\in J}$. Now define integration over the metric graph $\mathcal{G}$ as the sum of the integrals over all edges of a graph, namely for any integrable function $f=(f_j)_{j\in J}$ defined on $\mathcal{G}$
\begin{equation}\label{eq:int_G}
    \int_{\mathcal{G}}f(x)dx=\sum_{j\in J}\int_0^{l_j}f_j(x)dx.
\end{equation}
The weak solution $u$ should satisfy the condition
\begin{equation}\label{eq:weak}
  \int_0^T  \int_{\mathcal{G}} \left( u \partial_t \phi + \frac{u^2 \partial_x}{2}\phi\right) dxdt =
    \int_{\mathcal{G}} \mathring{u}(x)\phi(x,0)dx,
\end{equation}
for some $t\in[0,\infty)$.  Let us put our attention on the definition of the integral over $\mathcal{G}$. To pass from \eqref{eq:weak} to the strong from of the equation we  put the $x$ derivative on the equation, namely, we consider
\begin{eqnarray}
    \int_{\mathcal{G}} \frac{u^2 \partial_x\phi}{2} dx&=& 
    \sum_{j\in E} \left(\left.\frac{ u_j^2\phi_j}{2}\right|_{x=l_j}-\left.\frac{ u_j^2\phi_j}{2}\right|_{x=0}
    -\int_{[0,l_j]}u_j\partial_x u_j\phi_j\right)\\ \label{eq:parts}
    &=&-\int_{\mathcal{G}} u\partial_x u\phi+\sum_{i\in I}\left(\sum_{\left\{j\in J:\, e_j\in D_i^{in}\right\}}\left.\frac{ u_j^2\phi_j}{2}\right|_{x=l_j}-\sum_{\left\{j\in J:\, e_j\in D_i^{out}\right\}}\left.\frac{u_j^2\phi_j}{2}\right|_{x=0}\right).
\end{eqnarray}
So  to eliminate  boundary terms at each vertex $v_i$, using the Def. \ref{def:smooth}ii) for $k=0$, we require that
\begin{equation}\label{eq:Kirch}
 \sum_{\left\{j\in J:\, e_j\in D_i^{in}\right\}}\frac{u_j^2(l_j,t)}{2}\quad =\sum_{\left\{j\in J:\, e_j\in D_i^{out}\right\}}\frac{u_j^2(0,t)}{2}\qquad\text{for almost all}\,\, t\in[0,\infty).
\end{equation}
Equation \eqref{eq:Kirch}, known as the Kirchhoff condition, is one of the most classical transmission conditions considered on metric graphs, see \cite[Sec.~2.2.1]{Mugnolo2014}. It describes the conservation of flux in each vertex of a network.

\begin{df}\label{def:conserv}
We say that system \eqref{eq:weak}--\eqref{eq:trans+}--\eqref{eq:trans-} is the weak formulation of Burgers' equation on the metric tree $\mathcal{G}$, if weighted adjacency matrices of a line graph $\mathcal{B}^{pq}$, $p,q=0,1$, satisfy conditions \eqref{eq:n_return},\eqref{eq:Kirch}.
The class of solutions to the problem in weak  formulation we denote by $B(\mathcal{G})$.
\end{df}

The hyperbolic character of  Burgers' equation makes  determine the behaviour at vertices  to obtain the transmission condition for incoming characteristics, i.e. the coefficients of matrices $\mathcal{B}^{01}$ and $\mathcal{B}^{10}$. 
%In generic case, as shock waves are not appearing, we have flows coming into the vertex, say $v_i$, on edges $e_s \in \mathcal{C}^{in}_i(t)$ and flows coming out onto edges $e_s \in \mathcal{C}^{out}_i(t)$. Hence there is a need to give a receipt of how the part $\mathcal{C}^{in}_i$ determines the origins of edges from $\mathcal{C}^{out}_i$. 
In our setting we are obliged to take into account two restrictions. The first one is the Kirchhoff condition (\ref{eq:Kirch}) while the second is the requirement that dynamic on graph $\mathcal{G}$ is acyclic, namely \eqref{eq:n_return}. Note that the determination of a solution, even under the above restrictions, is not unique. To make the solver of our equation on $\mathcal{G}$ well posed, there is a need to impose more conditions. The general case is rather complex, so in this paper we concentrate  on two examples: the equation with non-negative velocities, and the general velocities on the honeycomb tree, see Def. \ref{def:p+h}. In the last case, the geometry of vertices is simple enough to consider all possible flow variations in vertices. It also gives some intuitions for the more general case.

The article is organised as follows. Section \ref{sec:posit_NBE} concentrates on non-negative case. The coefficients of $\mathcal{B}^{10}$ are related with the change of energy of the solution, see Subsection \ref{sec:posit_BC}, while the existence result in Theorem \ref{thm:p_exist} is derived using methodology known from Hamilton-Jacobi equation, it is our first main result. In Section \ref{sec:gen_NBE} general velocities on honeycomb trees are considered. The generalisation of energy methods applied to the vertices of the first and second kind, see Definition \ref{def:p+h}, with arbitrary direction of a flow in vertex can be found in Section \ref{sec:gen_trans} while the existence result in Subsection \ref{sec:gen_exist}, the second main result is stated as Theorem \ref{thm:main}. Finally, in Section \ref{sec:concl} we refer to the motivating example of wave interference. %and consider a few more examples explaining the behaviour of the solution to network Burgers' equation compared to mono-dimensional case.   

%\oadd{In the first case the direction of a flow agrees with the parameterization of all vertices which allows to transform the problem into a combination of problems on path graphs, see Def. \ref{def:p+h}. Consequently the problem can be interpreted as a combination of mono-dimensional cases. In Subsection \ref{sec:posit_exist}, using methodology known from Hamilton-Jacobi equation, see \cite[Sec.3.3]{Evans}, an explicit, unique solution that refers to well-established Lax-Oleinik formula is derived. Furthermore, Subsection \ref{sec:posit_BC} relates matrices $\mathcal{B}^{01}$, $\mathcal{B}^{10}$ in conditions \eqref{eq:trans+}--\eqref{eq:trans-} with the energy change of solution.}

\section{Non-negative entropy solutions}\label{sec:posit_NBE}

In this section the analysis is restricted to the flow direction that agrees with the parameterization of edges. Consequently, we look for weak solutions such that for $\mathring{u}>0$ the solution remains in the non-negative cone, $u\geq 0$. Considerations in the Section \ref{sec:posit_BC} relate coefficients of $\mathcal{B}^{01}(u)$ with some properties of the solution $u$ while in Section \ref{sec:posit_exist} we derive the existence theorem to the problem of a form
\begin{subequations}\label{eq:p_Burgers}
\begin{align}\label{eq:p_Burgers_e}
   \sum_{j\in J}\int_0^T  \int_{0}^{l_j} \left(u_j \partial_t \phi_j + \frac{u_j^2 \partial_x\phi_j}{2}\right) dxdt&=
    \sum_{j\in J}\int_{0}^{l_j} \mathring{u}_j(x)\phi_j(x,0)dx,\\[.2cm] \label{eq:p_Burgers_ic}
u_j(x,0)&=\mathring{u}_j(x)>0, \qquad x\in[0,l_j],\,j\in J,\\[.1cm] \label{eq:p_Burgers_bc} 
 u_j(0,t)&=\sum_{\left\{s\in J:\,e_s\in D_i^{in}\right\}}\,b^{01}_{js}(u) u_s(1,t),\qquad \text{for}\,\,\phi_{ij}^-\neq 0, \\[.1cm] \label{eq:p_Burgers_Kc}
     \sum_{\left\{j\in J:\, e_j\in D_i^{in}\right\}} u_j^2(l_j,t)&=\sum_{\left\{j\in J:\, e_j\in D_i^{out}\right\}} u_j^2(0,t),\qquad \text{for almost all}\, t\in[0,T].
\end{align}
\end{subequations}
Before we go through the details let us formalise the notion of non-negative solution.

\begin{df}\label{def:def_sws}
We say that function $u$ is a non-negative weak solution of network Burgers' equation \eqref{eq:p_Burgers} if 
\begin{itemize}
\item[(i)] $t\mapsto u_j(\cdot,t)\in L^{\infty}([0,l_j],\RR)$ is continuous almost everywhere on $[0,T)$, for $T>0$;
\item[(ii)] for every $\phi(\cdot,t) \in C^{\infty}(\mathcal{G})$ 
%with compact support contained in $[0,1]\times [0,T]$ 
$u$ satisfies \eqref{eq:p_Burgers_e},
\item[(iii)] $u\geq 0$ for every $\mathring{u}\in L^{\infty}([0,1],\mathbb{R}_+^m)$,
\item[(iv)] $u$ satisfies transmission conditions \eqref{eq:p_Burgers_bc} -- \eqref{eq:p_Burgers_Kc}.
\end{itemize}
\end{df}

\subsection{Derivation of transmission conditions}\label{sec:posit_BC}

The aim of this part is to understand how to derive coefficients of matrix $\mathcal{B}^{01}(u)$ in \eqref{eq:p_Burgers_bc}, hence in the whole Subsection \ref{sec:posit_BC} referring to the network Burgers' equation we consider the problem 
\begin{equation} \label{eq:p_net_Burgers}
\eqref{eq:p_Burgers_e}\,\text{--}\,\eqref{eq:p_Burgers_ic}\,\text{--}\,\eqref{eq:p_Burgers_Kc}.
\end{equation}

We learn from the mono-dimensional case that  to obtain the uniqueness of  weak solutions there is a need to specify the shock wave by \emph{Rankine-Hugoniot condition} and exclude non-physical shocks by, for instance, \emph{Lax condition}. Namely, let $\xi:[0,T)\rightarrow \mathbb{R}_+$ be a smooth curve describing the discontinuity of scalar weak solution $u$, and by $\xi^{\pm}(t)$ denote left and right limit when $x$ goes to $\xi(t)$. Then 
\begin{equation}\label{eq:RH&L}
    %\dot{s}(t)=\frac{f(u_L)-f(u_R)}{u_L-u_R}, \qquad \text{and}\qquad f'(u_L)>f'(u_R).
    \frac{d}{dt}\xi(t)=\frac{u(\xi^-(t),t)+u(\xi^+(t),t)}{2}, \qquad \text{and}\qquad u(\xi^-(t),t)>u(\xi^+(t),t).
\end{equation}

\begin{df}\label{def:def_ses}
We say that function $u_j: [0,l_j]\times [0,T)\rightarrow \RR$ is \emph{an entropy solution of scalar Burgers' equation on edge $e_j$}, $i=1,\ldots,m$ if it is a weak solution of to scalar Burgers' equation on edge $e_j$ which satisfies both Rankine-Hugoniot and Lax conditions at each discontinuity. \newline
Furthermore, $u=(u_j)_{j\in J}$ is \emph{edge-entropy solution} if it is an entropy solution at each edge.  
\end{df}

Let us remind also that in the mono-dimensional case \emph{Oleinik's one-sided inequality}
\begin{equation}\label{eq:O_inequal}
    u(x_2,t)-u(x_1,t)\leq \frac{x_2-x_1}{t},\qquad \text{for}\,\,x_1\leq x_2,\quad t>0.
\end{equation}
implies that $u$ is an entropy solution.

We concentrate on vertices now. Note first that Kirchhoff condition in vertex $v_i$ being resp. a source or a sink assures unique representation of solution $u_j(v_i,t)=0$, for $e_j\in D^{out}_i$ and $e_j\in D^{in}_i$ resp., since there is no flow through these vertices. In the case of other vertices we may obtain the ambiguity. Consequently, imposing only conditions \eqref{eq:RH&L} on the non-negative weak solution to \eqref{eq:p_net_Burgers} still does not guarantee the uniqueness. Let us stop at this statement for a moment. In order to define the fraction of mass that flows through the vertex $v_i$ at some fixed time $t$, let us transform a classical notion of Riemann solver into the transmission in the vertex counterpart. Denote by $u_j(v_i,t^\mp)$ the value of solution (in a head or a tail of an edge, respectively for $\phi_{ij}^-\neq 0$ and $\phi_{ij}^+\neq 0$), before the flow through the vertex for $t^-$ and after the flow for $t^+$.

\begin{df}
Let %$\mathcal{G}_i$ be a $v_i$-subgraph of a metric graph 
$\mathcal{G}=((V,E,\mathcal{L},\phi), d)$ be a metric graph and fix $v_i\in V$. %see Definition \ref{def:v-subgraph}. 
We say that a mapping 
\bd
\it{TS}_i: [0,\infty]^{\text{deg}(v_i)} \rightarrow [0,\infty]^{\text{deg}(v_i)},\qquad
 u(x,t^-)|_{J_i} \mapsto u(x,t^+)|_{J_i}, 
\ed
where $J_i$ is defined in \eqref{eq:G_i}, is a \emph{transmission solver in vertex $v_i\in V$}, if it satisfies conditions \eqref{eq:p_Burgers_Kc} for almost all $t\in[0,T)$. 
\end{df}

The first peculiarity implied by assuming only the Kirchhoff conditions in vertices is the lack of condition that joins values of solution before and after the flow through the vertex, namely at $t^-$ and $t^+$. 

\begin{exam}\label{exam:uniq_path}
Let $P_2$ be a path graph, see Definition \ref{def:p+h}, and consider a Riemann problem on metric path graph $\mathcal{P}_2$, presented in Figure \ref{fig:circle}, of the form
\be
\begin{array}{rcll}
\displaystyle \sum_{j=1}^2\int_0^T  \int_{0}^1 \left( u_j \partial_t \phi_j + \frac{u_j^2 \partial_x\phi_j}{2}\right) dxdt&=&
   \displaystyle  \sum_{j=1}^2\int_{0}^{1} \mathring{u}_j(x)\phi_j(x,0)dx,&\\
\mathring{u}_1(x)=0,\phantom{ii}\qquad \mathring{u}_2(x)&=&1,&x\in [0,1],\\
u_1(0,t)=u_2(1,t)=0,\qquad u_1^2(1,t)&=&u_2^2(0,t), &t\geq 0.\end{array}
\ee
The transmission solver $TS_2$ does not have to be unique at vertex $v_2=e_1(1)=e_2(0)$, for some neighbourhood of $t=0$. Note that for any parameter $a\in[0,\infty)$, $u$ defined below is a non-negative, edge-entropy solution for some $t\in[0,\epsilon)$. 
\begin{enumerate}
\item Let $a\in [0,1)$, then
\be \nonumber
u_1(x,t)&=&\left\{\begin{array}{ll}
0&\text{for }\,\, x\neq 1,\\
a&\text{for }\,\, x= 1,\end{array}\right. \\ \nonumber
u_2(x,t)&=&\left\{\begin{array}{ll}
a&\text{for }\,\, \frac{x}{t}\leq a,\\
\frac{x}{t}&\text{for }\,\, a< \frac{x}{t}\leq 1,\\
1&\text{for }\,\, \frac{x}{t}>1.\end{array}\right.
\ee
\item Let $a\in [1,\infty)$, then 
\be \nonumber
u_1(x,t)&=&\left\{\begin{array}{ll}
0&\text{for }\,\, x\neq 1,\\
a&\text{for }\,\, x= 1,\end{array}\right. \\ \nonumber
u_2(x,t)&=&\left\{\begin{array}{ll}
a&\text{for }\,\, \frac{x}{t}<\frac{a+1}{2},\\
1&\text{for }\,\, \frac{x}{t}>\frac{a+1}{2}.\end{array}\right.
\ee
\end{enumerate}
Obviously, each coordinate of $u$ is a piece-wise continuous solution to mono-dimensional Burgers' equation and at each jump satisfies Rankin-Hugoniot and Lax conditions. Consequently, by \cite[Thm.~4.2]{Bres2005} $u_j$ is an entropy solution of scalar Burgers' equation on edges $e_j$, $j=1,2$, so the edge-entropy solution to network Burgers. Finally, we derive a family of transmission solvers in $v_2$ at $t=0$, that depends on parameter $a$.
\begin{equation}
    TS_2(0,1)=(a,a), \qquad a\in[0,\infty).
\end{equation}
\end{exam}

Considerations on a path graph allow us to build the intuition related with the behaviour in vertices, as the solution can be easily related with the scalar case. Let us refer to the solutions presented in Example \ref{exam:uniq_path} with a standard solution of initial-boundary value problem on the interval $[0,2]$. Namely, with a problem of a form
\begin{equation*}
\begin{array}{rcll}
\displaystyle  \int_0^T  \int_{0}^2 \left( u \partial_t \phi + \frac{u_j^2 \partial_x\phi}{2}\right) dxdt&=&
\displaystyle  \int_{0}^{2} \mathring{u}(x)\phi_j(x,0)dx,&\\[.1cm]
\mathring{u}(x)&=&\begin{cases}0&x\in [0,1],\\1&x\in [1,2],\end{cases}&\\
u(0,t)&=&u(2,t)=0&t\geq 0.\end{array}
\end{equation*}
The comparison clearly indicates that  to obtain an entropy solution in a mono-dimensional case we need  to take $a=0$, since otherwise we introduce a non-physical shock into the model. The choice of $a\in(0,1]$ gives  a weak solution that can be justified, while $a>1$ seems to make no sense.  To choose a physically reasonable solution in the network case, we assume the continuity at some edges adjacent to the vertex $v_i$, a.e. in time. %Namely, 
%\begin{eqnarray} \label{eq:cont_in}
%    \text{continuity at the edges from}\,\,D_i^{in}\,\,\text{if the flow agrees with a direction of a vertex}\,\, v_i,\\ \label{eq:cont_out}
%    \text{continuity at the edges from}\,\, D_i^{out}\,\, \text{if the flow agrees with the opposite direction to}\,\, v_i.
%\end{eqnarray}
%For the case of non-negative solution, condition \eqref{eq:cont_in} simplifies to left-continuity, namely
Namely, continuity at the edges from $D_i^{in}$ if the flow agrees with the direction of a vertex. In the case of non-negative solution, this condition simplifies to
\begin{enumerate}
\item[\textit{(LC)}] \hspace{1cm} $u_j(1,t^-)=u_j(1,t^+),\qquad \text{for}\,\,e_j\in D_i^{in}$ and a.e. $t\in(0,T)$.
\end{enumerate}
Condition \textit{(LC)} transfers the problem of finding a value of solution at $t^+$ only into edges from $D_i^{out}$. It is worth mentioning that it is well defined only for vertices different than sinks. For the path graph, see Example \ref{exam:uniq_path}, it is sufficient to obtain the uniqueness; but not in the general case $\text{deg}_{-}(v_i)>1$. The next condition relates the value of solution after the flow through the vertex with the change of the energy, which is a natural assumption in the context of fluid-type equations.

Let us remind that in the case of scalar Burgers' equation the change of energy of piece wise continuous solution with one jump, defined on the interval $[A,B]$ reads
\begin{eqnarray}\label{eq:E'}
\frac{d}{dt}E(t)=\frac{u^3(A,t)}{3}-\frac{u^3(B,t)}{3}-\frac{(u(\xi^-(t),t)-u(\xi^+(t),t))^3}{12},
\end{eqnarray}
where $u(s^{\pm}(t),t)$ is the right and left limit at discontinuity curve $s$. We easily note that for each shock wave that satisfies the Lax condition, energy decreases proportionally to the magnitude of a jump, while for non-physical shocks we observe the increase of the energy. In the following consideration we take into account only edge-entropy solutions which implies that instead of non-physical shock waves we chose the rarefaction wave both in the interior of an edge and its head/tail. Since it is a continuous solution, it does not change the energy point-wise and the last entry in \eqref{eq:E'} vanishes. Now fix the vertex $v_i$ and consider the Riemann problem, at $x=1$ for incoming edges and $x=0$ for outgoing ones, that arises due to the flow through the vertex. We define the change of the energy at $v_i$ by $\mathcal{E}_i:[0,\infty)^{\text{deg}(v_i)}\rightarrow \mathbb{R}$
%\begin{equation*}
%	\begin{array}{lcl}\mathcal{E}_i(u(v_i,t))&=&\sum_{j:\,e_j\in D_i^{in}} \frac{u_j^3(1,t^-)-u_j^3(1,t^+)}{3} -\frac{\left(u_j(1,t^-)-u_j(1,t^+)\right)^3}{12}\theta\left(u_j(1,t^-)- u_j(1,t^+)\right), \\ [.2cm]
%	&+&\sum_{j:\,e_j\in D_i^{out}} \frac{u_j^3(0,t^+)-u_j^3(0,t^-)}{3} -\frac{\left(u_j(0,t^+)-u_j(0,t^-)\right)^3}{12}\theta\left(u_j(0,t^+)- u_j(0,t^-)\right),
%\end{array}
%	\end{equation*}
	\begin{equation} \label{eq:Energy}
	\begin{array}{lcl}
	\mathcal{E}_i(u(v_i,t))&=& \displaystyle  \sum_{j:\,e_j\in D_i^{in}} \mathcal{E}_{ij}^+(u(1,t))+\sum_{j:\,e_j\in D_i^{out}} \mathcal{E}_{ij}^-(u(0,t)), \\ [.5cm]
	\mathcal{E}_{ij}^{\pm}(u(v_i,t))&=&  \displaystyle  \frac{u_j^3(v_i,t^{\mp})-u_j^3(v_i,t^{\pm})}{3} -\frac{\left(u_j(v_i,t^{\mp})-u_j(v_i,t^{\pm})\right)^3}{12}\theta\left(u_j(v_i,t^{\mp})- u_j(v_i,t^{\pm})\right),
\end{array}
	\end{equation}
where $\mathcal{E}_{ij}^{\pm}:[0,\infty)^{\text{deg}(v_i)}\rightarrow \mathbb{R}$ is the change of energy at the edge $e_j$ and $\theta$ is a Heaviside step function. The following transmission conditions are related to extremes of $\mathcal{E}_i$.  

\begin{enumerate}
\item[\textit{($\mathcal{E}_i^m$)}] transmission conditions \eqref{eq:p_Burgers_bc} in $v_i$ minimize function $\mathcal{E}_i$,
\item[\textit{($\mathcal{E}_i^M$)}] transmission conditions \eqref{eq:p_Burgers_bc} in $v_i$ maximize function $\mathcal{E}_i$. 
\end{enumerate}

At the beginning let us remark that without condition \textit{(LC)} the problem of minimization of $\mathcal{E}_i$ with respect to $u(v_i,t^+)$ does not have to be well-posed. Let us return to the Example \ref{exam:uniq_path}. For $v_2$, at $t=0$, we have
\begin{equation*}
 \text{min}_{a\in [0,\infty)}\,\, \mathcal{E}_2(0,1,a,a)=-\infty,
\end{equation*}
since $\mathcal{E}_2$ reads
\bd
\mathcal{E}_2(0,1,a,a)=\left\{\begin{array}{ll}
-\frac{1}{3},&\text{for }a\in[0,1)\\[0.2cm]
-\frac{(a-1)^3}{12}-\frac{1}{3}&\text{for }a\in [1,\infty).\end{array}\right.
\ed
On the contrary maximizing $\mathcal{E}_i$ we obtain $a=1$ which is again not the solution we head to. In order to build further intuition we consider a problem defined on the metric honeycomb tree. 

\begin{exam}\label{exam:uniq_honey}
Let us consider metric honeycomb tree $\mathcal{H}_3$ being $v$-subgraph of honeycomb lattice for $v$ being a vertex of the first kind, see Figure \ref{fig:T1&2}(iii). Define on $\mathcal{H}_3$ a network Burgers' equation %\eqref{eq:p_Burgers}
\eqref{eq:p_net_Burgers} with initial condition $\mathring{u}(x):=(a,1,1)^T$, $a\in[0,1]$. 
The edge-entropy solution which satisfies condition \textit{(LC)} depends on one parameter $b\in [0,a]$, for $t\in[0,\epsilon)$, and reads
\be \nonumber
u_1(x,t)&=&\,\,a,\\ \nonumber
u_2(x,t)&=&\left\{\begin{array}{ll}
b&\text{for }\,\, \frac{x}{t}\leq b,\\
\frac{x}{t}&\text{for }\,\, b< \frac{x}{t}\leq 1,\\
1&\text{for }\,\, \frac{x}{t}>1,\end{array}\right. \\ \nonumber
u_3(x,t)&=&\left\{\begin{array}{ll}
\sqrt{a^2-b^2}&\text{for }\,\, \frac{x}{t}\leq \sqrt{a^2-b^2},\\
\frac{x}{t}&\text{for }\,\, \sqrt{a^2-b^2}< \frac{x}{t}\leq 1,\\
1&\text{for }\,\, \frac{x}{t}>1.\end{array}\right.
\ee
Now we build two transmission solvers which satisfy either \textit{($\mathcal{E}_i^m$)} or \textit{($\mathcal{E}_i^M$)}, and denote them respectively by $TS_2^m$ and $TS_2^M$. Function $\mathcal{E}_2$ is, for $t=0$, formulated by
\begin{equation*}
    \mathcal{E}_2\left(a,1,1,a,b,\sqrt{a^2-b^2}\right)=\frac{b^3+(a^2-b^2)^{\frac{3}{2}}-2}{3}
\end{equation*}
Calculating critical points of $\mathcal{E}_2$ and values at the boundary we arrive at three possible cases, namely $b=0$, $b=\frac{\sqrt{2}}{2}a$ and $b=a$. We note that 
\bd
\mathcal{E}_2(a,1,1,a,a,0)=\mathcal{E}_2(a,1,1,a,0,a)=\frac{a^3-2}{3} \quad \text{and} \quad \mathcal{E}_2\left(a,1,1,a,\frac{\sqrt{2}}{2}a,\frac{\sqrt{2}}{2}a\right)=\frac{\sqrt{2}a^3-4}{6},
\ed
hence $TS_2^m(a,1,1)=\left(a,\frac{\sqrt{2}}{2}a,\frac{\sqrt{2}}{2}a\right)$ and $TS_2^M(a,1,1)\in\left\{(a,0,a),(a,a,0)\right\}$.
\end{exam}

Example \ref{exam:uniq_honey} is very specific since the value of solution before the flow through the vertex is equal at $e_2$ and $e_3$, see Figure \ref{fig:T1&2} for the notation. Consequently, for $t=0$ edges $e_2$ and $e_3$ can be considered as locally symmetric with respect to the flow. In order to exclude such case in further considerations we introduce some technical condition called \emph{decreasing flow with respect to edge enumeration}
\begin{enumerate}
\item[\textit{(DF)}] \hspace{1cm}
$(TS_i^M)_j\geq (TS_i^M)_k$ for any $j<k$, $j,k\in D_i^{out}$. 
\end{enumerate}
%\begin{equation}\
%    (TS_i^M)_j\geq (TS_i^M)_k,\qquad \text{for any }j<k,\quad j,k\in D_i^{out}. 
%\end{equation}
It allows specifying the solution in which the highest flow is related to the edge with the lowest number. Since all tree graphs $G$ having the same triplet  $(V,E,\mathcal{L})$ but different mappings $\Phi$ that all satisfy an increasing order of edges are homomorphic, then any locally symmetric solution can be chosen depending on the choice of representative. In particular, using the notation introduced in Example \ref{exam:uniq_honey}, assuming that $TS_2$ satisfies \textit{(DF)} we have that $TS_2^M(a,1,1)=(a,a,0)$.

What happens if edges $e_2$ and $e_3$ are not locally symmetric with respect to the flow? We expect that it leads to different mass distribution when going through the vertex, depending on the value of $\mathring{u}_{2}$ and $\mathring{u}_{3}$. In such a case, coefficients of matrix $\mathcal{B}^{01}(u)$ in equation \eqref{eq:p_Burgers_bc} depend strictly on solution $u$. On the other hand, it is worth to underline that the considered transmission solver works point-wise in time and seems justified to add a consistency condition that allows it to stabilize, on a certain time interval. Namely, we expect that
\begin{equation}\label{eq:RS_consist}
    TS_i\left(TS_i\left(u(v_i,t^-)\right)\right)=u(v_i,t^+).
\end{equation}
Condition \eqref{eq:RS_consist} was also introduced in \cite[Def.~5]{GarPic2009} as one of common assumptions imposed on different transmission solvers considered in the literature. In line with this reasoning, let us define \emph{minimal} and \emph{ maximal transmission solver in vertex} as follows.  

\begin{df}\label{def:mM_RS}
Let $TS_i^m$ (resp. $TS_i^M$) be the transmission solver that, for some fixed $t\in[0,T)$, satisfy conditions \textit{(LC)}--\textit{($\mathcal{E}^m_i$)} (resp. \textit{(LC)}--\textit{($\mathcal{E}^M_i$)}--\textit{(DF)}) in $v_i$. We say that $(TS_i^{m})^{\star}$  (resp. $(TS_i^{M})^{\star}$) is a \emph{minimal (resp. maximal) transmission solver in vertex $v_i$} if it satisfies
\begin{equation}\label{eq:RS_lim}
    (TS_i^{z})^{\star}u(v_i,t^-)=\lim_{n\rightarrow \infty} (TS_i^z)^{(n)}u(v_i,t^-),\qquad \text{for any}\,\, u(v_i,t^-)\in [0,\infty)^{\text{deg}(v_i)},\, z=m,M,
\end{equation}
where, by $(TS_i^z)^{(n)}$, we understand the $n$-th composition of the mapping $TS_i^z$.
\end{df}

We need to justify now that Definition \ref{def:mM_RS} is well-posed, hence that the limit in \eqref{eq:RS_lim} exists. If it does not depend on $u$ 
then problem \eqref{eq:p_Burgers} transforms into
\begin{subequations}\label{eq:p_Burgers2}
\begin{align}\label{eq:p_Burgers_e2}
   \sum_{j\in J}\int_0^T  \int_{0}^{l_j} \left( u_j \partial_t \phi_j + \frac{u_j^2 \partial_x\phi_j}{2}\right) dxdt&=
    \sum_{j\in J}\int_{0}^{l_j} \mathring{u}_j(x)\phi_j(x,0)dx,\\[.2cm] \label{eq:p_Burgers_ic2}
u_j(x,0)&=\mathring{u}_j(x)>0, \qquad x\in[0,l_j],\,j\in J,\\[.1cm] \label{eq:p_Burgers_bc2} 
 u_j(0,t)&=\sum_{\left\{s\in J:\,e_s\in D_i^{in}\right\}}\,b^{01}_{js}(u) u_s(1,t),\qquad \text{for}\,\,\phi_{ij}^-\neq 0,
\end{align}
\end{subequations}
where coefficients of $\mathcal{B}^{01}$ in \eqref{eq:p_Burgers_bc2} are given by
\begin{equation}\label{eq:p_Burgers2_coef}
    b_{js}^{01}(u)=\frac{(TS_i^{z})^{\star}_j(u)}{\sum_{\left\{k\in J:\,e_k\in D_i^{in}\right\}}(TS_i^{z})^{\star}_k(u)}, \qquad \text{for}\quad z=m,M.
\end{equation}

%Unfortunately, even such a simple Example like \ref{exam:uniq_honey} indicates that flow symmetric edges are indistinguishable in the sense of maximizing the change of the energy. We should not expect the uniqueness of $(TS_i^M)^{\star}$, but what we obtain is the uniqueness with respect to equivalence relation $\sim_{fsym}$.

 %Theorem \ref{thm:RS_unique} shows that even taking different values of solution on outgoing edges, we arrive at symmetric division of mass in the case of minimization of function $\mathcal{E}_i$  \oadd{at the time interval}. Similarly in the case of function maximization \oadd{????????? co gdy maksymalizujemy?} 
 
%Nevertheless, the local choice of the solution can change the solution globally. 

\begin{thm}\label{thm:RS_unique}
Consider non-negative weak solution $u$ of Burgers' equation \eqref{eq:p_net_Burgers} on the metric tree $\mathcal{G}$ and fix $t\in[0,T)$. The following statements hold.
\begin{enumerate}
\item[(i)] At each vertex $v_i\in V$, there exists a unique transmission solver $(TS_i^m)^{\star}$ of the form
\begin{equation}\label{eq:RS_min}
    (TS_i^m)^{\star}u(v_i,t^-)=\left\{\begin{array}{ll}
     u_j(1,t^-)&\text{for}\,\, e_j\in D_i^{in},\\
    \frac{1}{\sqrt{\text{deg}(v_i)}}\sqrt{\sum_{\left\{s\in J:\,\,e_s\in D_i^{in}\right\}}u_s^2(1,t^-)}\qquad&\text{for}\,\, e_j\in D_i^{out}.
    \end{array}\right.
\end{equation}
\item[(ii)] At each vertex $v_i\in V$, there exists a unique transmission solver $(TS_i^M)^{\star}$ of the form
\begin{equation}\label{eq:RS_max}
    (TS_i^M)^{\star}u(v_i,t^-)=\left\{\begin{array}{ll}
     u_j(1,t^-)&\text{for}\,\, e_j\in D_i^{in},\\[.2cm]
    \sqrt{\sum_{\left\{s\in J:\,\,e_s\in D_i^{in}\right\}}u_s^2(1,t^-)}\qquad&\text{for}\,\, e_j=e_k,\\ [.2cm]
    0&\text{for}\,\, e_j\in D_i^{out}\setminus\left\{e_k\right\},
    \end{array}\right.
\end{equation}
where $k\in J$ satisfies condition 
\begin{equation}
    k:=\max\left\{j\in J:\,\, e_j\in D_i^{out}\right\}.
\end{equation}
\end{enumerate}
\end{thm}

\begin{proof}
Let $v_i\in V$ be an arbitrary vertex. % and denote by $l,l^{\pm}$ respectively $\text{deg}(v_i)$, $\text{deg}^{\pm}(v_i)$. 
Without loose of generality we assume that 
\begin{equation}\label{eq:tech_as1}
    \sum_{\left\{j\in J:\,\,e_j\in D_i^{in}\right\}}\,\,u_j^2(1,t^{-})=1,
\end{equation}
and introduce a notation $f^{\mp}=(f_j^\mp)_{j=1}^{\text{deg}(v_i)}:=(u_j^2(v_i,t^{\mp}))_{j=1}^{\text{deg}(v_i)}$. By \textit{(LC)}, finding $TS_i^z$, $z=m,M$, is equivalent to the optimization problem
\be\label{eq:min}
\begin{array}{c}
\bar{\mathcal{E}}(f^+)=\sum_{j=1}^{\text{deg}_+(v_i)} h_j\left(\sqrt{f_j^+}\right)\longrightarrow \min/\max, \\[.2cm]
\text{on the set}\,\, A=\left\{f^+\in [0,1]^{\text{deg}_+(v_i)}:\,\, \sum_{j=1}^{\text{deg}_+(v_i)} f_j^+=1\right\},\end{array}
\ee
%\be\label{eq:min1}
%\begin{array}{c}
%\sum_{s=1}^l g_s(\bar{u}_s)\longrightarrow \min, \\[.2cm]
%\text{on the set}\,\, A=\left\{(\bar{u}_1,\ldots,\bar{u}_l)\in [0,\infty)^l:\,\, \sum_{s=1}^l \bar{u}^2_s=C\right\},\end{array}
%\ee
where 
\be h_j(u)=\begin{cases}\frac{1}{3}\left(u^3-(f_j^-)^{\frac{3}{2}}\right)&\text{for}\,\,u^2<f_j^-,\\
\frac{1}{4}\left(u^3+(f_j^-)^{\frac{1}{2}}u^2-f_j^-u-f_j^{\frac{3}{2}}\right)&\text{for}\,\,u^2\geq f_j^-.
\end{cases}
\ee
Since we optimize a continuous function $\bar{\mathcal{E}}$ on a compact set $A$, the only thing to prove is the uniqueness of the existing minimum/maximum. We show that $\bar{\mathcal{E}}$ is strictly quasiconvex on a convex set $A$ and therefore attains a unique global minimum. Note that function $\lambda \mapsto \bar{\mathcal{E}}(\lambda f^++(1-\lambda)g^+)$, for $f^+,g^+\in A$ and $\lambda\in [0,1]$ is convex since
\begin{eqnarray*}
    &&\frac{d}{d\lambda^2}\bar{\mathcal{E}}(\lambda f^++(1-\lambda)g^+)=\\
    &&\phantom{xxxx}\quad \sum_{j=1}^{\,\,l^-}\frac{(f_j^+-g_j^+)^2}{2(\lambda f_j^++(1-\lambda)g_j^+)}\left(\left.\frac{d}{du^2}h_j(u)\right|_{u=\sqrt{\lambda f^+_j+(1-\lambda)g^+_j}}-\frac{\left.\frac{d}{du}h_j(u)\right|_{u=\sqrt{\lambda f^+_j+(1-\lambda)g^+_j}}}{2\sqrt{\lambda f^+_j+(1-\lambda)g^+_j}}\right)>0.
\end{eqnarray*}
Hence, it attains maximum at the boundary and 
\begin{equation}\label{eq:q-conv}
   \bar{\mathcal{E}}(\lambda f^++(1-\lambda)g^+)\leq \max\left(\bar{\mathcal{E}}( f^+),\bar{\mathcal{E}}(g^+)\right). 
\end{equation}
Since the inequality \eqref{eq:q-conv} is strict for $\lambda\in(0,1)$, $\bar{\mathcal{E}}$ is strictly quasiconvex. 

Using the methods of quasiconvex programming we know that maximum is attained at the boundary of $A$, see \cite[Lem.~3.2]{Flores2015}. Adding condition \textit{(DF)} we have a uniqueness of $TS_i^M$. 

We now derive the formula for $(TS_i^z)^{\star}$, $z=m,M$, starting with minimization condition. The idea is to describe sequences $(u(v_i,t_n^-))_{n\in \mathbb{N}}$ and $(u(v_i,t_n^+))_{n\in \mathbb{N}}$ in such a way that for  each step 
\begin{equation}
    u(v_i,t_{n+1}^-):=u(v_i,t_{n}^+), \qquad \text{and}\quad t_1:=t.
\end{equation} 
All velocities are non-negative, so for the next time step we obtain such regulation for the velocities coming out the chosen vertex. Let us fix arbitrary $n\in \mathbb{N}$ and denote by $u^-$ and $u^+$ the value of the solution in vertex $v_i$ in the time step $t_n$.  
\begin{eqnarray} \label{eq:u1}
    &u^-:=\left((TS_i^m)^{(n-1)}u(v_i,t^-)\right)_k, \qquad u^+:=\left((TS_i^m)^{(n)}u(v_i,t^-)\right)_k; \\
    &\text{where}\,\, k\in \left\{j\in J:\,\,e_j\in \text{max}_{e_j\in D_i^{out}}\,\, (TS_i^m)^{(n-1)}u(v_i,t^-) \right\}.
\end{eqnarray}
Now consider some index $s\in J$ such that 
\begin{eqnarray}\label{eq:u2}
 &\bar{u}^-:=\left((TS_i^m)^{(n-1)}u(v_i,t^-)\right)_s<u^-, \quad\text{and}\quad  \bar u^+:=\left((TS_i^m)^{(n)}u(v_i,t^-)\right)_s>\bar u^-.
\end{eqnarray}
Without loss of generality assume that the flow through the vertex $v_i$ in $t_n$ changes only values at two coordinates of edges adjacent to $v_i$. Since, for almost all $t$, Kirchhoff condition needs to be satisfied we have
\begin{equation}\label{eq:conserv}
    (u^-)^2+(\bar u^-)^2=(u^+)^2+(\bar u^+)^2.
\end{equation}
We show that the choice of transmission conditions described in \eqref{eq:u1}--\eqref{eq:u2} minimizes the function $\mathcal{E}_i$. Consequently, only the value given in \eqref{eq:RS_min} can be the limit $(TS_i^m)^{\star}$. 

Indeed, for $h>0$ and
$$
\bar u^+=\bar u^-+h, \mbox{ \ we have by \eqref{eq:conserv} \ } u^+=\sqrt{(u^-)^2 - 2(\bar u^-) h -h^2}.
$$
The structure of the data implies that 
\begin{eqnarray*}
    \mathcal{E}_i(u(v_i,t_n))&=&\sum_{j:\,e_j\in D_i^{out}\setminus \left\{e_k,e_s\right\}} \mathcal{E}_{ij}^-(u(0,t_n))+\frac{(\bar u^+)^3-(\bar u^-)^3}{3}-\frac{(\bar u^+-\bar u^-)^3}{12} +
    \frac{(u^+)^3-(u^-)^3}{3} \\
    &=&\sum_{j:\,e_j\in D_i^{out}\setminus \left\{e_k,e_s\right\}} \mathcal{E}_{ij}^-(u(0,t_n))+\frac{(\bar u^-+h)^3-(\bar u^-)^3}{3} - \frac{h^3}{12}\\ [0.1cm]
    &+&\frac{ \left((u^-)^2 - 2(\bar u^-) h -h^2\right)^{\frac{3}{2}}-(u^-)^3}{3}=:\tilde{ \mathcal{E}}(h)
\end{eqnarray*}
But then we note that
\begin{equation}\label{eq:E'h}
\frac{d}{dh}\tilde{\mathcal{E}}(h)|_{h=0} = (\bar u^-)(\bar u^- -u^-) <0.
\end{equation}
Hence $\mathcal{E}_i$ decreases locally with a growth of $h>0$. 

Let us turn now to the energy maximization case. Since the above considerations are working for $h$ negative also, the form of the derivative in \eqref{eq:E'h} ensures that the maximum is realised at the boundary of the set $A$. Condition \textit{(DF)} provides a final formula for $(TS_i^M)^{\star}$.

%We note that $g_s$ is an increasing strictly convex function on $[0,\infty)$ and the square root is quasi-convex. Consequently, for any $\lambda\in[0,1]$ and $(f_1,\ldots,f_l),(\bar{f}_1,\ldots,\bar{f}_l)\in[0,\infty)^l$ we have
\end{proof}

The assumptions of Theorem \ref{thm:RS_unique} are strictly related to non-negative velocities of flow. In the general case the considerations are more subtle and generate a larger number of possibilities of physical behaviour of a flow. For that reason in Section \ref{sec:gen_trans} we confine ourselves to honeycomb trees. Note that this metric graph provides only three types of transmission conditions, according to the formula \eqref{eq:p_Burgers2_coef}. Two for the vertices $v_i$ of the first kind, such that $D_i=\left(\left\{e_j\right\},\left\{e_k,e_l\right\}\right)$, $j<k$
\begin{enumerate}
    \item[(i)] $u_k(0,t)=u_j(1,t),\,\, u_k(0,t)=0$ 
    \item[(ii)] $u_j(0,t)=u_k(0,t)=\frac{\sqrt{2}}{2}u_j(1,t)$;
\end{enumerate}
and one for the vertices of the second kind such that $D_i=\left(\left\{e_j,e_k\right\},\left\{e_l\right\}\right)$
\begin{enumerate}
    \item[(iii)] $u_l(0,t)=\frac{u_j(1,t)}{\sqrt{u_j^2(1,t)+u_k^2(1,t)}}u_j(1,t)+\frac{u_k(1,t)}{\sqrt{u_j^2(1,t)+u_k^2(1,t)}}u_k(1,t)$.
\end{enumerate}

At the end of this part let us give the formal definition of entropy solution of network Burgers' equation. 

\begin{df}\label{def:}
We say that function $u: [0,1]\times [0,T)\rightarrow \RR^m$ is \emph{a vertex-entropy solution}, if it is a weak solution to network Burgers' equation \eqref{eq:p_Burgers}.

Furthermore, $u=(u_j)_{j\in J}$ is \emph{entropy solution} if it is an edge- and vertex-entropy solution. In particular \emph{minimal- and maximal-entropy solutions} are respectively the edge-entropy solutions to \eqref{eq:p_Burgers2} -- \eqref{eq:p_Burgers2_coef} with $z=m,M$.   
\end{df}

\subsection{Existence of solution}\label{sec:posit_exist}

We are finally ready to prove the existence result in the case of non-negative solutions. 
\begin{thm}\label{thm:p_exist}
Problem \eqref{eq:p_Burgers} for a finite tree $\mathcal{G}$ admits a non-negative entropy solution for any $\mathring{u}\in L^{\infty}([0,1],\mathbb{R}_+^m)$. For almost all $t>0$ function $x\mapsto u(x,t)$ has a locally bounded total variation and can be calculated recursively from the formula
\begin{subequations}\label{eq:explicitB}
\begin{eqnarray}\label{eq:explicitB1}
u_{j}(x,t)&=&\frac{x-y_{j}(x,t)}{t}, \qquad \qquad\text{where $y_j$ minimizes function}\\ \label{eq:explicitB2}
y\mapsto G_{j}(x,t,y)&=&\left( \frac{(x-y)^2}{2t}+\int_{0}^y\mathring{u}_j(s)ds\right)\chi_{[0,x]}(y)\\ \label{eq:explicitB3}
&+&\left(\frac{x(x-y)}{2t}-\int_{0}^{\frac{-y}{x-y}t}\frac{u_j^2(0,s)}{2}ds\right)\chi_{(-\infty,0)}(y),
\end{eqnarray}
\end{subequations}
for any edge $j\in J$.
\end{thm}

\begin{proof}
In accordance with the proof of existence of a weak solution in the scalar case, see \cite[Thm.1.1]{Lefloch2002}, we show that formula \eqref{eq:explicitB} is valid for piece-wise smooth solutions satisfying Lax shock inequality at discontinuity. To this end we define a solution recursively at each edge.

\smallskip

\emph{Necessity.} Assume first that $u$ is a solution of \eqref{eq:p_Burgers} as stated above. Then for any source $e_j(0)$, $j=1,\ldots,s$, see Subsection \ref{sec:graph}, the right hand side of \eqref{eq:p_Burgers_bc} vanishes and consequently $u_j(0,t)= 0$ for all $t>0$. Note that due to the tree structure and recursive procedure, we can choose such an order of edges that before calculating the solution on $k$-th edge we have values of all $u_j(x,t)$ for $j\in J$ such that $b_{kj}> 0$, see equation \eqref{eq:adj}. Consequently, the system of conservation laws on network transforms into a sequence of initial-boundary-value problems of a form
\begin{subequations}\label{eq:bv_Burgers}
\begin{align}\label{eq:bv_Burgers_e}
   \int_{0}^{l_j} \left( u_j \partial_t \phi_j + \frac{u_j^2 \partial_x \phi_j}{2} \right) dxdt&=
    \int_{0}^{l_j} \mathring{u}_j(x)\phi_j(x,0)dx,\\[.2cm] \label{eq:bv_Burgers_ic}
u_j(x,0)&=\mathring{u}_j(x)>0, \qquad x\in[0,l_j],\,j\in J,\\[.1cm] \label{eq:bv_Burgers_bc} 
 u_j(0,t)&=\sum_{\left\{s\in J:\,e_s\in D_i^{in}\right\}}\,b^{01}_{js}(t) u_s(1,t),
\end{align}
\end{subequations}
where $i\in I$ satisfies $v_i=e_j(0)$.

Let us fix $j\in J$ and define auxiliary function $w_j:[0,l_j]\times [0,\infty)\rightarrow \mathbb{R}_+$; $\mathring{w}_j:[0,l_j]\rightarrow \mathbb{R}_+$ such that 
\begin{equation}\label{eq:w}
w_{j}(x,t)=\int_{0}^x u_{j}(s,t)ds,\qquad \mathring{w}_j(x):=w_j(x,0).
\end{equation}
Note that $u$ is a weak, piece-wise smooth solution to \eqref{eq:p_Burgers} if and only if it satisfies 
\begin{equation*}
    \partial_t u_j+u_j\partial_x u_j=0
\end{equation*}
at each smoothness region in $[0,l_j]\times[0,T)$ and Rankine-Hugoniot condition along the discontinuity, see \cite[Thm.~2.3]{Lefloch2002}. We have 
\begin{eqnarray}\nonumber
   \int_{0}^x \partial_t u_j(s,t)+\partial_s\,\frac{u_j^2(s,t)}{2}ds &=& \partial_t w_j(x,t)+\frac{u_j^2(x,t)-u_j^2(0,t)}{2}\\ \label{eq:w_est}
    &=&\partial_t w_j(x,t)+\frac{(\partial_x w_j(x,t))^2}{2}-\left.\frac{(\partial_x w_j(x,t))^2}{2}\right|_{x=0}=0.
\end{eqnarray} 
By the properties of a square function we have that for any $v\in [0,\infty)$ and $z\in \mathbb{R}$
\bd
vz-\frac{v^2}{2}\leq\frac{z^2}{2}.
\ed
For $z=\partial_x w_j$, by \eqref{eq:w_est},
\begin{eqnarray*} \label{eq:est}
    v \partial_x w_j-\frac{v^2}{2}\leq \frac{(\partial_x w_j)^2}{2}=\left.\frac{(\partial_x w_j)^2}{2}\right|_{x=0}-\partial_t w_j,
\end{eqnarray*}
and consequently,
\begin{equation}\label{eq:charakt}
    \partial_t w_j+v\partial_x w_j \leq \frac{v^2}{2}+\left.\frac{(\partial_x w_j)^2}{2}\right|_{x=0}.
\end{equation}

In order to determine the value of $u_j$ at $(x,t)\in [0,l_j]\times [0,\infty)$ we chose some $v$. The line passing through $(x,t)$ with slope $v$ either intersects the ox axis at $y=x-vt\in [0,l_j]$ or hits the oy axis at $t=-\frac{y}{v}$ for $y<0$. We integrate \eqref{eq:charakt} along the characteristic $y=x-vt$ separately in two mentioned cases. 

If $y\in[0,l_j]$, then integrating over $[0,t]$, analogously to the proof in the scalar case, we have 
\begin{eqnarray} \nonumber
w_j(x,t)&\leq& \frac{v^2t}{2}+\int_0^t \frac{(\partial_x w_j(0,s))^2}{2} ds+\mathring{w}_j(x)\\ \label{eq:est_charakt+}
&=& \frac{(x-y)^2}{2t}+\int_0^t \frac{u_j^2(0,s)}{2}ds+\int_0^x\mathring{u}_j(s)ds.
\end{eqnarray} 
If $y\in(-\infty,0)$, then we integrate over $\left[-\frac{y}{v},t\right]$ and since $v=\frac{x-y}{t}$ and by \eqref{eq:w}, we obtain
\begin{eqnarray} \nonumber
w_j(x,t)&\leq& \frac{v^2}{2}\left(t+\frac{y}{v}\right)+\int_{-\frac{y}{v}}^t \frac{(\partial_x w_j(0,s))^2}{2}ds+w_j\left(0,-\frac{y}{v}\right)\\ \label{eq:est_charakt-}
&\leq & \frac{x(x-y)}{2t}+\int_{-\frac{yt}{x-y}}^t \frac{u_j^2(0,s)}{2}ds.
\end{eqnarray}
Finally, by \eqref{eq:est_charakt+} and \eqref{eq:est_charakt-} 
\be\label{eq:min}
w_j(x,t)\leq \int_0^t \frac{u_j^2(0,s)}{2}ds +G_j(x,t,y),
\ee
where $G_j$ is defined in \eqref{eq:explicitB2} -- \eqref{eq:explicitB3}. Since the left hand side does not depend on $y$ we minimize the right hand side over $y$. Let us choose the slope of the characteristic line $v=u_j(x,t)$. Inserting $v$ to \eqref{eq:charakt} we obtain, by \eqref{eq:w_est}, the equality. The minimum in \eqref{eq:min} is attained for $y_j$ since $u$ satisfies Lax condition. Finally,
\begin{eqnarray}
w_j(x,t)&=&\int_0^t \frac{u_j^2(0,s)}{2}ds+G_j(x,t,y_j), \qquad \text{where}\quad y_j:=\text{arg min}_{y\in \mathbb{R}}\,G_j(x,t,y),%\qquad\text{where}\\
%G_j(x,t,y)&=&\left( \frac{(x-y)^2}{2t}+\int_{0}^y\mathring{u}_j(s)ds\right)\chi_{[0,x]}(y)\\
%&+&\left(\frac{x(x-y)}{2t}-\int_{0}^{\frac{-y}{x-y}t}\frac{u_j^2(0,s)}{2}ds\right)\chi_{(-\infty,0)}(y),
\end{eqnarray}
and since $y_j(x,t)=x-u_j(x,t)t$ we derive a formula \eqref{eq:explicitB1}.

\smallskip 

\emph{Sufficiency.} Assume that $u$ is given by \eqref{eq:explicitB}, and show that it is a weak solution to \eqref{eq:p_Burgers}. Note firstly that $u$ is well defined since for any $j=1,\ldots,m$, there exists a unique minimizer of $G_j$. 

The existence of a minimizer of $G_j$ for $y\in[0,x]$ is obvious since the first entry in \eqref{eq:explicitB2} grows faster than linearly while the second has at most linear growth. The same argument works in the case $y\in(-\infty,0)$ if we transform the problem of minimiaztion of $G_j$ over $y_j$ into minimization of function $H_j:[0,x]\times (0,\infty)\times [0,\infty)\rightarrow \mathbb{R}$,
    \begin{equation}\label{eq:def_H}
        H_j(x,t,\tau_j):=\frac{x^2}{2(t-\tau_j)}-\int_{0}^{\tau}\frac{u_j^2(0,s)}{2}ds,
    \end{equation}
    over $\tau_j$, where
%Introduce now auxiliary parameter  $\tau_j(x,t)$ and function $H_j:[0,x]\times (0,\infty)\times [0,\infty)\rightarrow \mathbb{R}$
    \begin{eqnarray}\label{eq:def_tau}
    \tau_j(x,t)&:=&\frac{-y_j(x,t)}{x-y_j(x,t)}t, \qquad \text{for some }y_j(x,t)<0.%\\
    %H_j(x,t,\tau_j)&:=&G_j(x,t,y_j)=\frac{x^2}{2(t-\tau_j)}-\int_{0}^{\tau}\frac{u_j^2(0,s)}{2}ds.
    \end{eqnarray}
    %Minimization of $G_j$ over $y_j$ is equivalent with minimization 
    %We transform the problem into equivalent one, namely we show that the assumption that
    %\begin{eqnarray*}
    %x \mapsto y_j(x,t)=-\frac{\tau_j(x,t)}{t-\tau_j(x,t)}x,\qquad \text{where }\tau(x,t):=\text{arg min}_\tau\left(H_j(x,t,\tau)\right),
    %\end{eqnarray*}
    %is decreasing, which leads to the contradiction.

We show now that $x\mapsto y_j(x,t)$ is non-decreasing. Consequently it has locally bounded total variation and it is continuous in all but countably many points. It is sufficient for the uniqueness of the minimizer of $G_j$ and the well-posedness of $u$ for almost all $(x,t)$.  

Let us fix $t>0$; and by an abuse of notation denote by $y_1:=y_j(x_1,t)$, $y_2:=y_j(x_2,t)$ for any $x_1,x_2\in[0,x]$. Denote by $x_0\in[0,x]$ an argument such that $y_j(x_0,t)=0$. By contradiction we assume that $x\mapsto y_j(x,t)$ is decreasing and we consider three cases.

\begin{enumerate}
    \item %$0 \leq y_2<y_1\leq x_1<x_2$ 
    $0 \leq y_2<y_1$ and $x_0\leq x_1<x_2$
    
    From the definition of $y_1$, $G_j(x_1,t,y_1)\leq G_j(x_1,t,y_2)$. Additionally,
    %\begin{eqnarray*}
    %    x_1-y_2&=&\frac{x_1-x_2}{x_1-x_2+y_2-y_1}\left(x_1-y_1\right)+\frac{y_2-y_1}{x_1-x_2+y_2-y_1}\left(x_2-y_2\right)\\
    %    x_2-y_1&=&\frac{y_2-y_1}{x_1-x_2+y_2-y_1}\left(x_1-y_1\right)+\frac{x_1-x_2}{x_1-x_2+y_2-y_1}\left(x_2-y_2\right),
    %\end{eqnarray*}
    %By convexity a flux and Jensen inequality we have 
    \begin{equation*}
        \left(\frac{x_2-y_1}{t}\right)^2+\left(\frac{x_1-y_2}{t}\right)^2<\left(\frac{x_1-y_1}{t}\right)^2+\left(\frac{x_2-y_2}{t}\right)^2.
    \end{equation*}
    Finally, using \eqref{eq:explicitB2} we obtain the contradiction with the fact that $y_2$ minimizes $y\mapsto G_j(x_2,t,y)$ 
    \begin{eqnarray*}
    G_j(x_2,t,y_1)\leq G_j(x_1,t,y_2)-G_j(x_1,t,y_1)+G_j(x_2,t,y_1)<G_j(x_2,t,y_2).
    \end{eqnarray*}
    
    \item $y_2<y_1\leq 0$ and $ x_1<x_2<x_0$ 
    
    Using the notation in \eqref{eq:def_H} -- \eqref{eq:def_tau}, introduce $\tau_1:=\tau_j(x_1,t)$ and $\tau_2:=\tau_j(x_2,t)$. Conditions $y_2<y_1$ and $x_1<x_2$ imply that 
    %\begin{equation*}
    %    -\frac{\tau_2}{t-\tau_2}x_1<-\frac{\tau_2}{t-\tau_2}x_2<-\frac{\tau_1}{t-\tau_1}x_1,
    %\end{equation*} 
    $\tau_1<\tau_2$ and therefore we can repeat the reasoning in point 1. Again $H_j(x_1,t,\tau_1)\leq H_j(x_1,t,\tau_2)$ and
    %\begin{eqnarray*}
    %\frac{x_1}{t-\tau_1}&=&\frac{(x_1-x_2)(t-\tau_2)}{(t-\tau_1)x_1-(t-\tau_2)x_2}\frac{x_1}{t-\tau_2}+\frac{x_1(\tau_2-\tau_1)}{(t-\tau_1)x_1-(t-\tau_2)x_2} \frac{x_2}{t-\tau_1},\\
    %\frac{x_2}{t-\tau_2}&=&\frac{x_2(\tau_2-\tau_1)}{(t-\tau_1)x_1-(t-\tau_2)x_2}\frac{x_1}{t-\tau_2}+ \frac{(x_1-x_2)(t-\tau_1)}{(t-\tau_1)x_1-(t-\tau_2)x_2}\frac{x_2}{t-\tau_1},
    %\end{eqnarray*}
    %which leads to inequality
     \begin{equation*}
        \frac{x_1^2}{t-\tau_1}+\frac{x_2^2}{t-\tau_2}<\frac{x_1^2}{t-\tau_2}+\frac{x_2^2}{t-\tau_1}.
    \end{equation*}
    Using \eqref{eq:def_H}, we obtain the contradiction with the fact that $\tau_2$ minimizes $\tau_j \mapsto H_j(x_2,t,\tau)$ 
    \begin{eqnarray*}
    H_j(x_2,t,\tau_1)\leq H_j(x_1,t,\tau_1)-H_j(x_1,t,\tau_2)+H_j(x_2,t,\tau_1)<H_j(x_2,t,\tau_2).
    \end{eqnarray*}
    \item %$y_2<0\leq y_1\leq x_1<x_2$
    $y_2<0\leq y_1$ and $x_1<x_2$
    
    Note that $x\mapsto y_i$ is non-decreasing on both intervals $[0,x_0]$ and $[x_0,x]$ so consequently $x_0\leq x_1<x_2\leq x_0$, which leads to the contradiction. 
\end{enumerate}

%We conclude that $x\mapsto y_j(x,t)$ is almost everywhere monotone function, has locally bounded total variation and therefore is continuous in all but countably many points. Consequently the minimizer $\bar{y}_j$ of $G_j$ is uniquely defined. 

We show that \eqref{eq:explicitB} is a weak solution. Define now functions $a_{j\epsilon}, u_{j\epsilon}, f_{j\epsilon}, v_{j\epsilon}\in L^{\infty}([0,1]\times \mathbb{R}_+)$ such that
\begin{eqnarray*}
a_{j\epsilon}(x,t)&:=&\int_{-\infty}^0 e^{-\frac{1}{\epsilon}G_j(x,t,y)}dy+\int^{\infty}_0e^{-\frac{1}{\epsilon}G_j(x,t,y)}dy,\\ [.1cm]
%&=&\int_{0}^t\frac{xt}{2(t-\tau)^2}e^{-\frac{1}{\epsilon}H_j(x,t,\tau)}d\tau +\int^{\infty}_0 e^{-\frac{1}{\epsilon}G_j(x,t,y)}dy,\\ [.1cm]
u_{j\epsilon}(x,t)&:=&\frac{1}{a_{j\epsilon}(x,t)}\left(\int_{-\infty}^0\frac{2x-y}{2t}e^{-\frac{1}{\epsilon}G_j(x,t,y)}dy+\int^{\infty}_0\frac{x-y}{t}e^{-\frac{1}{\epsilon}G_j(x,t,y)}dy\right),\\[.1cm]
%&=&\frac{1}{a_{j\epsilon}(x,t)}\left(\int_{0}^t\frac{x^3t}{2(t-\tau)^3}e^{-\frac{1}{\epsilon}H_j(x,t,\tau)}dy+\int^{\infty}_0\frac{(x-y)^2}{2t^2}e^{-\frac{1}{\epsilon}G_j(x,t,y)}d\tau\right),\\[.1cm]\\[.1cm]
f_{j\epsilon}(x,t)&:=&\frac{1}{a_{j\epsilon}(x,t)}\left(\int_{-\infty}^0\frac{x(x-y)}{2t^2}e^{-\frac{1}{\epsilon}G_j(x,t,y)}dy+\int^{\infty}_0\frac{(x-y)^2}{2t^2}e^{-\frac{1}{\epsilon}G_j(x,t,y)}dy\right).%\\[.1cm]
%&=&\frac{1}{a_{j\epsilon}(x,t)}\left(\int_{0}^t\frac{x^3t}{2(t-\tau)^3}e^{-\frac{1}{\epsilon}H_j(x,t,\tau)}dy+\int^{\infty}_0\frac{(x-y)^2}{2t^2}e^{-\frac{1}{\epsilon}G_j(x,t,y)}d\tau\right),\\[.1cm]
\end{eqnarray*}
Set additionally 
\begin{equation}
    v_{j\epsilon}(x,t)=\log{a_{j\epsilon}(x,t)}.
\end{equation}
Note now that functions $(x,t) \mapsto G_j(x,t,y)$ and $(x,t) \mapsto v_{j \epsilon}(x,t,y)$, are differentiable with respect to $x$ and $t$; and hence $u_{j\epsilon}=-\epsilon\, \partial_t v_{j\epsilon}(x,t)$, $f_{j\epsilon}(x,t,y)=\epsilon\, \partial_x v_{j\epsilon}(x,t,y)$ we have
\begin{equation}\label{eq:weak2}
    \partial_t u_{j\epsilon}+\partial_x f_{j\epsilon}=0.
\end{equation}
%Now for $(x,t)$ such that $x\mapsto y_j(x,t)$ is continuous; there exists unique minimizer $y_j(x,t)$ of $y\mapsto G_j(x,t,y)$. 
We show that 
\be\label{eq:limit1}
\lim_{\epsilon \rightarrow 0^+}u_{j\epsilon}(x,t)=u_j(x,t)\quad \text{and}\quad \lim_{\epsilon \rightarrow 0^+}f_{j\epsilon}(x,t)=f_j(x,t),
\ee
for any $(x,t)$ in which $x\mapsto y_j(x,t)$ is continuous. Denote by $\bar{y}_j(x,t)$ the unique minimizer of $G_j$ at $(x,t)$ and define a mapping 
\bd
y\mapsto \bar{G}_j(x,t,y):=G_j(x,t,y)-G_j(x,t,\bar{y}_j);
\ed
which attains in $\bar{y}_{j}$ its minimum equal to $0$. Since $\bar{G}_j$ is locally Lipschitz continuous (which in particular on the interval $(-\infty,0)$ follows from reformulation \eqref{eq:def_H} -- \eqref{eq:def_tau}) then for any $\delta>0$, the estimate holds for $y\in [\bar{y}_j(x,t)-\delta,\bar{y}_j(x,t)+\delta]$ with Lipschitz constant $C_{j1}(x,t)$. Therefore 
\begin{eqnarray*}
a_{j\epsilon}(x,t)=\int_{\mathbb{R}}e^{-\frac{1}{\epsilon}\bar{G}_j(x,t,y)}dy &\geq& \int_{\bar{y}_j(x,t)-\delta}^{\bar{y}_j(x,t)+\delta}e^{-C_{j1}(x,t)|y-\bar{y}_j|}dy\\
&=&\frac{2}{C_{j1}(x,t)}\left(1-e^{-\frac{C_{j1}(x,t)\delta}{\epsilon}}\right)\epsilon \geq C_{j2}(x,t)\epsilon,
\end{eqnarray*}
for all $\epsilon<\delta$. On the other hand, for $y$ such that $|y-\bar{y}_j|\geq \delta$, $\bar{G}_j$ is bounded away from zero and attains infinity in infinity, hence
\be
e^{-\frac{1}{\epsilon}\bar{G}_j(x,t,y)}\leq e^{-\frac{1}{\epsilon}C_{j3}(x,t,\delta)|y-\bar{y}_j|}.
\ee
Finally we have
\begin{eqnarray} \nonumber
|u_{j\epsilon}-u_j|&=&\frac{1}{a_{j\epsilon(x,t)}t}\left(\int_{\left\{y:\,|y-\bar{y}_j|<\delta\right\}}|y-\bar{y}_j|e^{-\frac{1}{\epsilon}G_j(x,t,y)}dy+\int_{\left\{y:\,|y-\bar{y}_j|\geq \delta\right\}}|y-\bar{y}_j|e^{-\frac{1}{\epsilon}G_j(x,t,y)}dy\right)\\ \nonumber
%{\int_{\mathbb{R}}e^{-\frac{1}{\epsilon}G_j(x,t,y)}dy}\\ \nonumber
&\leq& \frac{\delta}{t}+\frac{2}{C_{j2} t \epsilon}\int_{0}^{\infty}ye^{-\frac{C_{j3}}{\epsilon}y}dy=\frac{\delta}{t}+\frac{2}{C_{j2}C_{j3}^2 t}\epsilon.
\end{eqnarray}
Passing to 0 with $\epsilon$ we receive the first limit in \eqref{eq:limit1}. Analogously, we calculate the second and passing with $\epsilon\rightarrow 0$ in \eqref{eq:weak2} we conclude that $u$ is a weak solution to \eqref{eq:p_Burgers}.

\eqref{eq:explicitB} is edge-entropy solution since it satisfies Oleinik's one-sided inequality \eqref{eq:O_inequal}. Indeed, by the fact that $x\mapsto y_j(x,t)$ is non-decreasing and positive, for any $x_1\leq x_2$, $x_1,x_2\in[0,l_j]$ and a.e. $t>0$
\begin{equation*} 
    u_j(x_2,t)-u_j(x_1,t)=\frac{x_2-y_j(x_2,t)}{t}-\frac{x_1-y_j(x_1,t)}{t}\leq \frac{x_2-x_1}{t}.
\end{equation*}
For the case with  negative $y's$ we get
\begin{equation*}
    u_j(x_2,t)-u_j(x_1,t)=\frac{x_2}{t-\tau(x_2,t)}-\frac{x_1}{t-\tau(x_1,t)}=\frac{x_2-x_1}{t-\tau(x_2,t)} + \left(\frac{x_1}{t-\tau(x_2,t)} -\frac{x_1}{t-\tau(x_1,t)}\right) \leq \frac{x_2-x_1}{t-\tau(x_2,t)}.
\end{equation*}
Since transmission conditions in \eqref{eq:p_Burgers_bc} are defined uniquely we arrive at an entropy solution.

Finally, on every edge the weak solution in piece-wise $C^1$ function so taking the limit $x_2-x_1\rightarrow 0$
\begin{equation}\label{eq:bound}
    \partial_x u_j(x_1,t)=\lim_{x_1-x_2\rightarrow 0}\frac{u_j(x_2,t)-u_j(x_1,t)}{x_2-x_1}\leq \max\{\frac{1}{t},\frac{1}{t-\tau(x_2,t)}\},
\end{equation}
we arrive with the estimate on $u_x$ at a.e. $(x_1,t)$.
\end{proof} 

Note that equations \eqref{eq:explicitB} are the counterparts of \emph{Lax-Oleinik formulas}, see \cite[Eq.~IV.1.3]{LeFloch1988}, for Burgers' equation on a tree. For the graph that satisfies condition 
\begin{equation}\label{eq:drzewo1}
    \text{deg}_+(v_i)\leq 1,\qquad  \text{for any}\,\, i\in I,
\end{equation}  
it is possible to relate this solution with the standard formulation on the straight line. The core property in this representation is to derive coefficients $\mathcal{B}^{01}(u)$ that are independent of a flow when we move back-word along the characteristic line. 

\begin{figure}[h]
\center
(i)\quad \begin{tikzpicture}

\def\tmp{1.7};

\foreach \x in {0,1}{\path[draw=gray,dashed] (0,\tmp*\x)
node[
regular polygon,%fill=gray!50,%wypelłnienie kolorem
draw=none,
regular polygon sides=6,
draw,
inner sep=.6cm,
] (hexagon) {};}	
					
\path[draw=gray, dashed] (0.985*3/2,\tmp/2)
node[
regular polygon,%fill=gray!50,%wypelłnienie kolorem
draw=none,
regular polygon sides=6,
draw,
inner sep=.6cm,
] (hexagon) {};	
	
\draw[->,thick] (-1/2,\tmp/2) -- node[below=2pt] {$e_1$} (1/2,\tmp/2);
\draw[->,thick] (1/2,\tmp/2) -- node[left=2pt] {$e_2$} (1,\tmp);
\draw[->,thick] (1/2,\tmp/2) -- node[right=1pt] {$e_3$} (1,0);
\draw (1/2,\tmp/2) node[right=2.5pt]{$v_0$};			
\end{tikzpicture} \qquad \qquad (ii)\quad \begin{tikzpicture}
\def\tmp{1.7};
\foreach \x in {0,1}{\path[draw=gray,dashed] (0,\tmp*\x)
node[
regular polygon,%fill=gray!50,%wypelłnienie kolorem
draw=none,
regular polygon sides=6,
draw,
inner sep=.6cm,
] (hexagon) {};}	
\path[draw=gray, dashed] (-0.985*3/2,\tmp/2) %\tmp/2
node[
regular polygon,%fill=gray!50,%wypelłnienie kolorem
draw=none,
regular polygon sides=6,
draw,
inner sep=.6cm,
] (hexagon) {};
\draw[->,thick] (-1/2,\tmp/2) -- node[below=2pt] {$e_3$} (1/2,\tmp/2);
\draw[->,thick] (-1,\tmp) -- node[left=2pt] {$e_1$} (-1/2,\tmp/2);
\draw[->,thick] (-1,0) -- node[left=1pt] {$e_2$} (-1/2,\tmp/2);
\draw (-1/2,\tmp/2) node[above right=2.5pt]{$v_0$};			
\end{tikzpicture}

	\caption{Two $v_0$-subgraphs of honeycomb tree, for $v_0$ being respectively a vertex of (i) a first kind (ii) a second kind. Illustration for Example \ref{exam:path} and considerations in Subsection \ref{sec:gen_trans}. }
		\label{fig:HT1&2}
\end{figure}
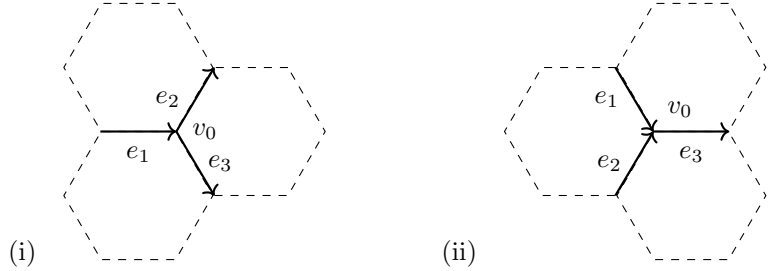

\begin{exam}\label{exam:path}
In order to explain this situation consider again two kinds of nodes in honeycomb tree, see Figure \ref{fig:HT1&2}(i)(ii) and the transmission conditions in vertex $v_0$ characterised by minimal transmission solver $(TS^m)^{\star}$. 
\begin{enumerate}
    \item[(i)] $v_0$ is of a first kind \newline 
    The idea now is to define the solution on the graph at point $(x,t)$ on the edge $e_3$ along the path $e_0e_1e_3$ where by $e_0$ we understand a half line $e_0=(-\infty,0)$ with initial condition $\mathring{u}_0=0$ and transmission conditions between edges $e_0$ and $e_1$ that conserve both the mass and the flux, namely
    \begin{equation}
        u_0(1,t)=u_1(0,t),\qquad \text{for almost all}\,\,t>0.
    \end{equation} 
    We change the reasoning in the proof of Theorem \ref{thm:p_exist} in the following way. Considering the characteristic line passing through $(x_0,t_0)$ with slope $v_0$ (assume $y_0=x_0-v_0t_0<0$) we allow it to go through the vertex and continue until it hits the initial line. Using the formula for transmission conditions \eqref{eq:p_Burgers_bc2} -- \eqref{eq:p_Burgers2_coef}, we conclude that characteristic line passes through the point $\left(1,-\frac{y_0}{v_0}\right)$ on the edge $e_1$ with a slope $\sqrt{2}v_0$. Then it intersects either $e_1$ or $e_0$ at $(y_1,0)$. Finally the explicit formula for the solution is given by 
    \begin{subequations}
\begin{eqnarray}
u_3(x_0,t_0)&=&\frac{x_0-y_3(x_0,t_0)}{t_0}\\ \nonumber
\text{where $y_3$ minimizes function}\phantom{xxxxx}&&\\
y\mapsto G_3(x_0,t_0,y)&=&\int_{-\infty}^{y_1} \mathring{u}_1(s)ds+\frac{(x_0-y)^2}{2t_0}.\phantom{xxxxxxxxxxx}
\end{eqnarray}
\end{subequations}
    \item[(ii)] $v_0$ is of a second kind \newline 
    The first problem in repeating the reasoning in (i) for $v_0$ is the lack of uniqueness of the path since we can chose either $e_0e_1e_3$ or $e_0e_2e_3$. The more essential problem, however, is the fact that we cannot define neither the slope of characteristic line $v_1$ on $e_1$, nor its counterpart on $e_2$ - $v_2$. Such representation does not result from transmission condition 
    \begin{equation*}
        v_0=\frac{\sqrt{2}}{2}\left(v_1+v_2\right).
    \end{equation*}
\end{enumerate}
\end{exam}

Example \ref{exam:path} indicates that the condition \eqref{eq:drzewo1} allows to choose the unique path from any point $x\in \mathcal{G}$ to the source and ensures well-posedness of the following procedure: $\mathcal{G}$ is a finite tree, so it is possible to re-enumerate edges in the way that for any two edges $e_{s}, e_{j}\in E$, and for any chosen path $e_{s}=e_{k_1},\ldots,e_{k_l}=e_{j}$; $k_i<k_{i+1}$ for all $i\in 1,\ldots,l-1$. Fix $e_j\in E$ and define a path $P_{j}=e_{k_1}e_{k_2}\ldots e_{k_{N_j}}$ of a length $L_j=\sum_{s=1}^{N_j} l_s$ that starts in a source and ends in $e_j$. Now define $u_{P_j}:(-\infty,l_j]\times [0,\infty)\rightarrow \mathbb{R}_+$ and $\mathring{u}_{P_j}:(-\infty,l_j]\rightarrow \mathbb{R}_+$ such that
\begin{eqnarray*}
u_{P_j}(x,t)&:=&\sum_{s=1}^{N_j}\, \left(\prod_{p=s}^{N_j-1}\frac{1}{b_{k_{p+1}k_{p}}^{01}}\right)\,\, u_{k_s}\left(x+\sum_{p=s}^{N_j-1}l_p,t\right)\, \chi_{\left(-\sum_{p=s}^{N_j-1}l_p, -\sum_{p=s+1}^{N_j-1}l_p\right]}(x),\\
\mathring{u}_{P_j}(x)&:=&u_{P_j}(x,0).
\end{eqnarray*}
%and note that for any $x\in (0,l_j]$
%\be\label{eq:equal}
%u_{P_j}(x,t)=u_{k_{N_j}}(x,t)=u_{j}(x,t).
%\ee

\begin{rem}
The solution to the problem \eqref{eq:p_Burgers} for a finite tree $\mathcal{G}$ that satisfies  \eqref{eq:drzewo1} can be related with mono-dimensional case using the counterpart of Lax-Oleinik formula on the path sub-graph, namely the formula for any $j\in J$ is given by
\begin{subequations}
\begin{eqnarray}\nonumber
u_j(x,t)&=&\frac{x-y_j(x,t)}{t}\\ \nonumber
\text{where $y_j$ minimizes function}\phantom{xxxxx}&&\\ \nonumber
y\mapsto G_j(x,t,y)&=&\int_{-\infty}^y \mathring{u}_{P_j}(s)ds+\frac{(x-y)^2}{2t}.\phantom{xxxxxxxxxxx}
\end{eqnarray}
\end{subequations}
\end{rem}
Finally, it is worth underlining that the considerations presented in the proof of Theorem \ref{thm:p_exist} can be generalised in the number of directions. Firstly, we can examine conservation law on the edges of a network, coupled by the linear transmission of mass that satisfies the conservation of flux condition for $f\in C^1([0,\infty))$ such that
\begin{equation}\label{eq:f}
f''>0,\qquad f(0)=0\qquad \text{and}\qquad \lim_{u\rightarrow \infty} \frac{f(u)}{u}=+\infty.    
\end{equation} 
On the other hand, we can introduce some sources of mass in vertices $v_i$ such that $\phi_{ij}^+=0$ for any $j\in J$.

We formalise those observations into a remark.

\begin{rem}
Let $f$ be a flux function that satisfies \eqref{eq:f}. For any $\mathring{u}\in L^{\infty}([0,1],\mathbb{R}_+^m)$ and $\bar{u}\in L^{\infty}([0,T],\mathbb{R}_+^m)$, the proof of Theorem \ref{thm:p_exist} can be repeated to the following generalisation of a problem \eqref{eq:p_Burgers}, for almost all $t\in[0,T]$,
\begin{subequations}\label{eq:p_gen}
\begin{align}\label{eq:p_gen_e}
   \sum_{j\in J}\int_0^T  \int_{0}^{l_j} \left(u_j \partial_t \phi_j + f(u_j )\partial_x\phi_j\right) dxdt&=
    \sum_{j\in J}\int_{0}^{l_j} \mathring{u}_j(x)\phi_j(x,0)dx,\\[.2cm] \label{eq:p_gen_ic}
u_j(x,0)&=\mathring{u}_j(x)>0, \qquad x\in[0,l_j],\,j\in J,\\[.1cm] \label{eq:p_gen_tc} 
 u_j(0,t)&=\sum_{\left\{s\in J:\,e_s\in D_i^{in}\right\}}\,b^{01}_{js}(u) u_s(1,t),\qquad \text{for}\,\,\phi_{ij}^-\neq 0,\,\,\text{deg}_+(v_i)>0 \\[.1cm] \label{eq:p_gen_Kc}
     \sum_{\left\{j\in J:\, e_j\in D_i^{in}\right\}} f(u_j(l_j,t))&=\sum_{\left\{j\in J:\, e_j\in D_i^{out}\right\}} f(u_j(0,t)),\qquad \text{for}\,\,\text{deg}_+(v_i)>0,\\[.1cm]\label{eq:p_gen_bc}
     u_j(0,t)&=\bar{u}_j(t)\geq 0,\qquad \text{for}\,\,\phi_{ij}^-\neq 0,\,\,\text{deg}_+(v_i)=0.
\end{align}
\end{subequations}
\end{rem}
\begin{proof}
The proof of this fact can be found in \cite[Thm.~2.1]{LeFloch1988}. %For the purpose of this paper we show only the estimate on $u_x$ that will be applied in further considerations.
%\ola{dopisac}
\end{proof}
\subsection{Dense subclass of positive solutions}

Note that for positive solutions one can distinguish a special class of functions which are preserved under the flow.  This class is the same as for the classical mono-dimensional Burgers' equation.
\begin{prop}\label{prop:W^+}
Let $\mathcal{G}$ be a metric tree. We introduce a class of functions $\mathcal{W}^+$ such that
\begin{multline}\label{eq:W+}
    f \in \mathcal{W}^+ \mbox{ \ \ iff \ \ } 
    \{ f \in B(\mathcal{G}): f \mbox{ is piece-wise 
    $C^1$ non-decreasing non-negative function,} \\
    \mbox{  the number of 
    jumps is finite and } 
    \mbox{ side derivatives exist at each point of $\mathcal{G}$} \}
\end{multline}

Then the class $\mathcal{W}^+$ is preserved by the flow generated by the Burgers' equation \eqref{eq:p_Burgers}, i.e. 
if $\mathring{u}\in \mathcal{W}^+$ then $u(t)\in \mathcal{W}^+$
for any $t>0$.
\end{prop}
\begin{proof}
Let $\mathring{u}\in \mathcal{W}^+$. Since in the interior of each edge we have the mono-dimensional 
situation then $\mathcal{W}^+$ class is preserved there. The only element that needs to be clarified is a transmission condition, namely that $u_{j}(0,t)$ is piece-wise $C^1$, non-decreasing with a finite number of jumps for every $j\in J$. %On the straight line, it is clear,  however in the graph case there is a need to look closer at dynamics at vertices in order to make sure that the number of discontinuities remain finite.
The properties of solution going out from an arbitrary vertex $v_i$ in the tree $\mathcal{G}$ can be considered as the composition of flows going out of two vertices $v_i'$ and $v_i''$ which are associated with $v_i$ by the following relation
\begin{equation}\label{eq:v'&''}
  \text{deg}_+(v_i')=\text{deg}_+(v_i),\quad  \text{deg}_-(v_i')=1;\qquad \text{and}\qquad \text{deg}_+(v_i'')=1,\quad  \text{deg}_-(v_i'')=\text{deg}_-(v_i).
\end{equation}
See also Figure \ref{fig:node_div}. We can easily see that vertex $v_i'$ joins the flow, while $v_i''$ splits it into outgoing edges. %On the honeycomb grid, for positive solutions, we have two type of vertices. The first one splits the flow, while the second one joints it. 
In the case of $v_i'$, the flow in $e(0)$ is a square root of sum of squared flows of incoming edges. Since on each $e_j\in D_i^{in}$ the flow is non-decreasing, $C^1$ function, then this properties are preserved for $e(0)$ for all but finite number of points. Since for $v_i''$ the transmission conditions at the head of outgoing edges are just proportions of the flow coming to the tail of edge $e$, fine properties are guaranteed.

%In the first case the boundary conditions at the head of an edge is just a proportion of the flow coming to the tail of the adjacent edge. It follows that the boundary condition is monotone non-increasing modulo finite number of jumps. Consequently, the value of a flow does not increase when it goes through the vertex in all but finitely many points. In the second case the flow at the head is a "sum" of flows in the tails of two incoming edges. Here again monotonicity of incoming solutions guarantees the fine properties of the boundary conditions. 

Finally, we note that the number of jumps can be multiplied by $\text{deg}_-(v_i)$ as the shock crosses $v_i$ but the finiteness of the graph ensures the control of the number of jumps. We shall also recall that under evolution some jumps may disappear.

\end{proof}

\begin{figure}
    \centering
    (i) \begin{tikzpicture}
\def\tmp{1.7};

\draw[->,thick] (-1/2,\tmp) node[below=2pt] {\scriptsize{$e_1$}} --  (1/2-2/6,\tmp/2+\tmp/6);
\draw[->,thick] (-2/3,\tmp/2) -- node[below=2pt] {\scriptsize{$e_2$}\phantom{xxx}} (1/2-1/3,\tmp/2);
\draw[->,thick] (-1/2,0) node[right=2pt] {\scriptsize{$e_3$}} --  (1/2-2/6,\tmp/2-\tmp/6);
\draw[->,thick] (1/2+1/6,\tmp/2+\tmp/6) node[right=2pt] {\scriptsize{$e_4$}} --  (1,\tmp);
\draw[->,thick] (1/2+1/6,\tmp/2-\tmp/6) node[right=2pt] {\scriptsize{$e_5$}} --  (1,0);

\draw (-2/3,\tmp/2) node[left=1pt] {$
\cdot$};	
\draw (1/2-1/3,\tmp/2) node[right=.5pt] {$v_i$};
\draw (1*10/11,\tmp*12/11) node[above=0.5pt, right=.25pt] {$\cdot$};
\draw (-1/2,\tmp*12/11) node[above=0.5pt, left=.25pt] {$\cdot$};
\draw (1*10/11,-\tmp*1/11) node[below=0.5pt, right=.25pt] {$\cdot$};
\draw (-1/2,-\tmp*1/11) node[above=0.5pt, left=.25pt] {$\cdot$};
\end{tikzpicture}\qquad (ii) \begin{tikzpicture}
\def\tmp{1.7};

\draw[->,thick] (-1/2,\tmp) node[below=2pt] {\scriptsize{$e_1$}} --  (1/2-2/6,\tmp/2+\tmp/6);
\draw[->,thick] (-2/3,\tmp/2) -- node[below=2pt] {\scriptsize{$e_2$}\phantom{xxx}} (1/2-1/3,\tmp/2);
\draw[->,thick] (-1/2,0) node[right=2pt] {\scriptsize{$e_3$}} --  (1/2-2/6,\tmp/2-\tmp/6);
\draw[->,thick] (1/2+1/3,\tmp/2) -- node[below=2pt] {\scriptsize{$e$}} (9/6,\tmp/2);
%\draw[->,thick] (1/2+1/6,\tmp/2+\tmp/6) node[right=2pt] {\scriptsize{$e_4$}} --  (1,\tmp);
%\draw[->,thick] (1/2+1/6,\tmp/2-\tmp/6) node[right=2pt] {\scriptsize{$e_5$}} --  (1,0);

\draw (-2/3,\tmp/2) node[left=1pt] {$
\cdot$};	
\draw (1/2-1/3,\tmp/2) node[right=.5pt] {$v_i'$};
%\draw (1*10/11,\tmp*12/11) node[above=0.5pt, right=.25pt] {$\cdot$};
\draw (-1/2,\tmp*12/11) node[above=0.5pt, left=.25pt] {$\cdot$};
%\draw (1*10/11,-\tmp*1/11) node[below=0.5pt, right=.25pt] {$\cdot$};
\draw (-1/2,-\tmp*1/11) node[above=0.5pt, left=.25pt] {$\cdot$};
\draw (9/6,\tmp/2) node[right=1pt] {$\cdot$};	
\end{tikzpicture} \qquad \begin{tikzpicture}
\def\tmp{1.7};

%\draw[->,thick] (-1/2,\tmp) node[below=2pt] {\scriptsize{$e_1$}} --  (1/2-2/6,\tmp/2+\tmp/6);
\draw[->,thick] (-2/3,\tmp/2) -- node[below=2pt] {\scriptsize{$e$}} (1/2-1/3,\tmp/2);
%\draw[->,thick] (-1/2,0) node[right=2pt] {\scriptsize{$e_3$}} --  (1/2-2/6,\tmp/2-\tmp/6);
\draw[->,thick] (1/2+1/6,\tmp/2+\tmp/6) node[right=2pt] {\scriptsize{$e_4$}} --  (1,\tmp);
\draw[->,thick] (1/2+1/6,\tmp/2-\tmp/6) node[right=2pt] {\scriptsize{$e_5$}} --  (1,0);

\draw (-2/3,\tmp/2) node[left=1pt] {$
\cdot$};	
\draw (1/2-1/3,\tmp/2) node[right=.5pt] {$v_i''$};
\draw (1*10/11,\tmp*12/11) node[above=0.5pt, right=.25pt] {$\cdot$};
%\draw (-1/2,\tmp*12/11) node[above=0.5pt, left=.25pt] {$\cdot$};
\draw (1*10/11,-\tmp*1/11) node[below=0.5pt, right=.25pt] {$\cdot$};
%\draw (-1/2,-\tmp*1/11) node[above=0.5pt, left=.25pt] {$\cdot$};
\end{tikzpicture}
    \caption{Transformation of arbitrary vertex $v_i$ in a tree $\mathcal{G}$, illustration (i), into two vertices $v_i'$, $v_i''$, illustration (ii), according to formula \eqref{eq:v'&''} introduced in Proposition \ref{prop:W^+}. }
    \label{fig:node_div}
\end{figure}
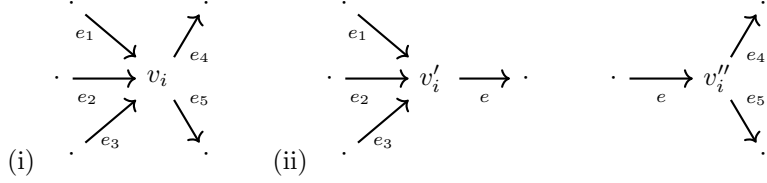

In further considerations we will use also $\mathcal{W}^+_{opp}$ class such that 
\begin{multline}\label{eq:W+_{opp}}
    f \in \mathcal{W}^+_{opp} \mbox{ \ \ iff \ \ } 
    \{ f \in B(\mathcal{G}): f \mbox{ is piece-wise 
    $C^1$ non-increasing non-negative function,} \\
    \mbox{  the number of 
    jumps is finite and } 
    \mbox{ side derivatives exist at each point of $\mathcal{G}$} \}
\end{multline}

\bigskip

\begin{df}
Let $u$ be a function defined over the graph $\mathcal{G}$. We say that $u\in TV(\mathcal{G})$ iff 
\begin{equation*}
    \|u\|_{TV(\mathcal{G)}}=\sum_{j\in J} %\int_{0}^{l_j}
    \|u_j\|_{TV(e_j)} \mbox{ \ is finite}.
\end{equation*}
\end{df}

Let us start with the estimates of $TV$-norm of non-negative solution for specified family of graphs that can be generalised for arbitrary metric trees. 

\begin{lem}\label{lem:p_tv-est}
Let $\mathcal{G}$ be a metric honeycomb tree with one source $e_1(0)$ and one sink $e_m(1)$. For $u$ being a solution to the problem \eqref{eq:p_Burgers} given by Theorem \ref{thm:p_exist}, the following estimate holds
\begin{equation}\label{tv-est}
    \sup_{t\in [0,T]} \|u^2(t)\|_{TV(\mathcal{G})}    +\int_0^T|\partial_t u_m^2|(1,t) dt \leq 2^{\kappa_\mathcal{G}}\left(\|\mathring{u}\|_{TV(\mathcal{G})} + \int_0^T |\partial_t u_1^2|(0,t)dt\right),
\end{equation}
where $\kappa_\mathcal{G}$ depends on graph structure. %Above we assume that the first $e_1$ edge initiates the graph, and $e_N$ is the only one "last" edge of the graph. For the general case we have a natural generalisation.
\end{lem}

\begin{proof}
Let us remind first that solutions $u$ given by Theorem \ref{thm:p_exist} are non-negative. We start with showing that $u\in \mathcal{W}^+$ restricted to an arbitrary edge $e_j$ satisfies 
\begin{equation}\label{tv1}
    \frac{d}{dt} \int_{e_j} |\partial_x u_{j}^2|dx + |\partial_t u_j^2|(1,t) \leq |\partial_t u_{j}^2|(0,t).
\end{equation}
If $u$ is from $\mathcal{W}^+$-class, then for $u_j$ there exists a finite sequence $0=\xi_0(t)<\xi_1(t) < ... < \xi_{K(t)}(t)=1$ for a.e. $t\in [0,T)$ such that
\begin{equation*}
    u_j(x,t)=\sum_{k=0}^{K(t)-1} u_j(x,t)\chi_{[\xi_k(t),\xi_{k+1}(t)]}(x),
\end{equation*}
and on each interval $[\xi_k(t),\xi_{k+1}(t)]$ $u$ is non-decreasing. We extended it by the left and right hand side limits. Furthermore,
$K(t)$ is piece-wise constant so there exists a finite sequence $0<t_0<t_1< ... < t_M<T$ such that $K(t)$ is constant over each interval $(t_i,t_{i+1})$. Note also $K(t)$ is decreasing 
as $u_1(0,t)=0$ by Kirchhoff condition. In accordance with previous notation we distinguish
the left and right limits at points $\xi_k$ by respectively $u(\xi^{\mp}_k(t),t)$. We have
\begin{multline}
    \frac{d}{dt} \int_{[0,1]} |\partial_x u_j^2| dx =
    \\
     \frac{d}{dt} \left[
    \sum_{k=0}^{K(t)-1} (u_j^2(\xi_{k+1}^-(t),t) - u_j^2(\xi_k^+(t),t)) {+
    \sum_{k=0}^{K(t)-2} (u_j^2(\xi_{k+1}^-(t),t) - u_j^{2}(\xi_{k+1}^+(t),t))} \right].
\end{multline}
Since, for a.e. $t\in (t_i,t_{i+1})$, $K(t)$ is constant, then for $0<k<K(t)$ %(for jumps)
\begin{equation*}
    \frac{d}{dt}   u_j(\xi_{k+1}^-(t),t) = \partial_t u_j(\xi_{k+1}^-(t),t) + \partial_x u_j(\xi_{k+1}^-(t),t) \frac{d\xi_{k+1}(t)}{dt}.
\end{equation*}
But by the Rankine-Hugoniot and Lax conditions for $\xi_{k+1}$, see \eqref{eq:RH&L}, %$\frac{d\xi_{k+1}(t)}{dt}= \frac12 (u(\xi_{k+1}^-(t),t)+u(\xi_{k+1}^+(t),t))$ and by definition $u(\xi_{k+1}^-(t),t)> u(\xi_{k+1}^+(t),t)$, so 
we conclude
\begin{equation*}
   \frac{d}{dt}   u^2_j(\xi_{k+1}^-(t),t) = u_j(\xi_{k+1}^-(t),t) \partial_x  u_j(\xi_{k+1}^-(t),t) 
   \left(u_j(\xi_{k+1}^+(t),t) - u_j(\xi_{k+1}^-(t),t)\right) \leq 0.
\end{equation*}
In the same manner we prove that
\begin{equation*}
   \frac{d}{dt}   u_j^2(\xi_{k+1}^+(t),t) \geq 0.
\end{equation*}
Taking into account the boundary terms coming from $k=0,K(t)$ we find that
\begin{equation*}
    \frac{d}{dt} \int_{[0,1]} |\partial_x u_j^2| dx\leq \partial_t u_j^2(1,t)-\partial_t u_j^2(0,t).
\end{equation*}
The boundary terms $-\partial_t u_j^2(z,t)=2u_j^2(z,t)\partial_x u_j(z,t)$, $z=0,1$ are non-negative since we are working in $\mathcal{W}^+$-class, which leads to (\ref{tv1}).

The class $\mathcal{W}^+$ is dense in $TV(\mathcal{G})$, so it allows us to approximate any $TV$-flow by an element from the $\mathcal{W}^+$--class. In order to pass to the limit we need to generate the global estimate, namely one which is independent of $K(t)$. Integrating (\ref{tv1}) over $(t_i,t_{i+1})$ we get
\begin{equation}
    \sup_{t\in [t_i,t_{i+1}]} \left( \|u_j^2\|_{TV(e_j)}(t)
    + \int_{t_i}^{t} |\partial_t u_j^2|(1,t)dt \right) \leq \int_{t_i}^{t_{i+1}} |\partial_t u_j^2|(0,t)dt
    + \|u_j^2(t_i)\|_{TV(e_j)}.
\end{equation}
Summing up over all intervals $(t_i,t_{i+1})$ we get
%\begin{equation}
%    \sup_{[t_{K(T)},T]} \|u_j^2\|_{TV(e_j)}
%    + \int_{0}^{T} |\partial_t u_j^2|(1,t)dt \leq \int_{0}^{T} |\partial_t u_j^2|(0,t)dt
%    + \|\mathring{u}^2_j\|_{TV(e_j)},
%\end{equation}
%and finally
\begin{equation}\label{eq:est1}
    \|u_j^2(T)\|_{TV(e_j)}
    + \int_{0}^{T} |\partial_t u_j^2|(1,t)dt \leq \int_{0}^{T} |\partial_t u_j^2|(0,t)dt
    + \|\mathring{u}_j^2\|_{TV(e_j)}.
\end{equation}
In order to make $TV$-norm of $u_j^2$ $T$-independent we transform \eqref{eq:est1} into 
%\padd{Tu jest kłopocik mamy, mamy 3 mozliwosci co mozemy dostac z tej nierownosci}
%\begin{equation}
%\max\left\{    \sup_{[0,T]} \|u^2\|_{TV(e_j)}
%    ; \int_{0}^{T} |\partial_t u^2|(1,t)dt \right\} \leq \int_{0}^{T} |\partial_t u^2|(0,t)dt
%    + \|u^2_0\|_{TV(e_j)}.
%\end{equation}
\begin{equation}\label{eq:est2}
    \sup_{[0,T]} \|u_j^2\|_{TV(e_j)}
    + \int_{0}^{T} |\partial_t u_j^2|(1,t)dt \leq 2\left( \int_{0}^{T} |\partial_t u_j^2|(0,t)dt
    + \|\mathring{u}_j^2\|_{TV(e_j)})\right).
\end{equation}
%\padd{kazda na swoj sposob jest dobra, każda jest zła .., ALE nie ZALEZy od T !!!}
Now we are ready to construct an approximation sequence tending to the desired solution for some general data $\mathring{u}_j\in TV(e_j)$. For given $\epsilon>0$ and $\bar{u}_j=u_j(0,\cdot) \in TV(0,T)$ we claim there exist 
\begin{equation}
    \mathring{u}_{j\epsilon} \in \mathcal{W}^+
    \mbox{ \ \ and \ \ } 
    \bar{u}_{j\epsilon} \in \mathcal{W}^+_{opp}
\end{equation}
such that
\begin{equation}
    \|\mathring{u}_{j\epsilon}^2 - \mathring{u}_j^2\|_{L^1(e_j)} < \epsilon
    \mbox{ \ \ and \ \ }
    \|\bar{u}_{j\epsilon}^2 - \bar{u}_j^2\|_{L^1(0,T)}
    < \epsilon
\end{equation}
and
\begin{equation}
    \|\mathring{u}_{j\epsilon}^2\|_{TV(e_j)}
    \leq \|\mathring{u}_j^2\|_{TV(e_j)}
    \mbox{ \ \ and \ \ }
    \|\bar{u}^2_{j\epsilon}\|_{TV(0,T)} \leq \|
    \bar{u}_j^2\|_{TV(0,T)}.
\end{equation}

So the considerations for the $u\in\mathcal{W}^+$ deliver the existence of $u_{\epsilon}$ solutions on the time interval $[0,T]$ and \eqref{eq:est2} implies the following estimate independent of $\epsilon$.
\begin{multline}
\qquad    \sup_{[0,T]} \|u^2_{j\epsilon}\|_{TV(e_j)}
    + \int_{0}^{T} |\partial_t u^2_{j\epsilon}|(1,t)dt \leq 2\left( \int_{0}^{T} |\partial_t u^2_{j\epsilon}|(0,t)dt
    + \|\mathring{u}^2_{j\epsilon}\|_{TV(e_j)}\right)
    \\
   \leq  2\left( \int_{0}^{T} |\partial_t \bar{u}_j^2|(0,t)dt
    + \|\mathring{u}_j^2\|_{TV(e_j)}\right).
\end{multline}
The above estimates imply the uniform bound for
\begin{equation}
    \partial_t u_\epsilon \in L^\infty(0,T;\mathcal{M}(e_j)).
\end{equation}
This leads, up to a subsequence $\epsilon\to 0$, to 
\begin{equation}
    \begin{array}{lcr}
    u_{j\epsilon} \to u_j & \mbox{in} & 
    L^1([0,T] \times e_j),\\
    u_{j\epsilon} \rightharpoonup^\ast  u_j & \mbox{in} &
    L^\infty([0,T]\times e_j),\\
    u_{j\epsilon}|_{x=l_j} \to u_j|_{x=l_j} & \mbox{in} &
    L^1(0,T).
    \end{array}
\end{equation}
In particular we have the point-wise convergence in the domain and at the boundary $\{x=1\}$. So we conclude $u$ is the solution to Burgers' equation at $e_j$.

\smallskip

Finally, to obtain $TV$-estimate \eqref{tv-est} for the whole graph we proceed recursively from the edge $e_j$ to the source $e_1(0)$, specifying the right hand side of (\ref{tv1}). $\mathcal{G}$ is a metric honeycomb tree so there is restricted number of vertices' types, see Definition \ref{def:p+h} and remarks below. 

Consider a vertex of the first kind $v_i$, and denote edges adjacent to it in the following way $D_i=(\left\{e_j\right\},\left\{e_k,e_l\right\})$. By the transmission conditions \eqref{eq:p_Burgers_Kc} 
\begin{equation}
    u_z^2(0,t)=\theta_z u_{j}^2(1,t) \mbox{ \ \ with \ \ }
    \theta_z \leq 1, \,\,z=k,l.
\end{equation}
Then by differentiation in time we find that 
\begin{equation}
    |\partial_t u_z^2|(0,t)=\theta_z |\partial_t u_{j}^2|(1,t),\quad z=k,l.
\end{equation}
So the identity (\ref{tv1}) gives a term on the left hand side which dominates the terms $|\partial_t u_{z}^2|(0,t)$, $z=k,l$, namely
\begin{equation}\label{eq:1k_est}
       \frac{d}{dt}\left( \int_{e_j} 2 |\partial_x u_{j}^2|dx + \int_{e_k} |\partial_x u_{k}^2|dx + \int_{e_l} |\partial_x u_{l}^2|dx\right)+ |\partial_t u_k^2 |(1,t) + |\partial_t u_l^2 |(1,t)\leq 
    2|\partial_t u_j^2|(0,t).
\end{equation}
Note that $v_i$ has two out-going edges and $\theta_z\leq 1$ for $z=k,l$, therefore, the equation for $e_{j}$ is taken twice. 

In the second case as $v_i$ is of the second kind, using notation $D_i=(\left\{e_j,e_k\right\},\left\{e_l\right\})$,  we have
\begin{equation}
    u_{j}^2(1,t)+u_{k}^2(1,t)=u_{l}^2(0,t).
\end{equation}
Then we easily deduce that
\begin{equation}\label{bo-3}
    |\partial_t u_l^2|(0,t) \leq | \partial_t u_{j}^2|(1,t)+
    | \partial_t u_{k}^2|(1,t),
\end{equation}
and analogously to \eqref{eq:1k_est} we obtain
\begin{equation}\label{eq:2k_est}
       \frac{d}{dt}\left( \int_{e_j} |\partial_x u_{j}^2|dx + \int_{e_k} |\partial_x u_{k}^2|dx + 2\int_{e_l} |\partial_x u_{l}^2|dx\right)+ |\partial_t u_l^2 |(1,t)\leq 
    |\partial_t u_j^2|(0,t)+|\partial_t u_k^2|(0,t).
\end{equation}
Finally taking the vertex from the path graph such that $D_i=(\left\{e_j\right\},\left\{e_k\right\})$ we have a conservation of mass in the vertex and consequently
\begin{equation}\label{eq:0k_est}
       \frac{d}{dt}\left( \int_{e_j} |\partial_x u_{j}^2|dx + \int_{e_k} |\partial_x u_{k}^2|dx \right)+ |\partial_t u_k^2 |(1,t)\leq 
    |\partial_t u_j^2|(0,t).
\end{equation}

Repeating iteratively above steps, and taking all edges with required multiplicity $\kappa_j$ that depends on the degree of vertex and its position in the graph we obtain
\begin{equation}
    \frac{d}{dt} \sum_{j\in J}\int_{e_j} 2^{\kappa_j}|\partial_x u_{j}^2|dx + |\partial_t u_N^2 |(1,t) \leq 
    2^{\kappa_1}|\partial_t u_1^2|(0,t),
\end{equation}
since the graph $\mathcal{G}$ has exactly one source $e_1(0)$ and one sink $e_m(0)$. After the integration by parts implies (\ref{tv-est}). %The above statement holds for the honeycomb, for the general case there is a need to change $"2"$ by the maximal $deg_{v_i}$.

%To finish the prove we shall note that the above estimate allows to pass to the limit (as we consider the 
%solutions from $\mathcal{W}$-class only). Getting the bound for any non-negative initial data in $TV(\mathcal{G})$. 
\end{proof}

\begin{rem}
Estimate derived in Lemma \ref{lem:p_tv-est} can be extended into arbitrary metric tree $\mathcal{G}$ having sources $e_j(0)$, $j=1,\ldots, s$ and sinks $e_j(l_j)$ for $j=m-S+1,\ldots, m$
\begin{equation}\label{tv-est-gen}
    \sup_{t\in [0,T]} \|u^2(t)\|_{TV(\mathcal{G})}    +\int_0^T\sum_{j=m-S+1}^m|\partial_t u_j^2|(l_j,t) dt \leq C_\mathcal{G}\left(\|\mathring{u}\|_{TV(\mathcal{G})} + \int_0^T\sum_{j=1}^s |\partial_t u_j^2|(0,t)dt\right).
\end{equation}

\end{rem}
\begin{proof}
The general case is slightly more involving. Assume that for vertex $v_i$
\begin{equation*}
    D_i^{in}=\{e_{k_1}, ... , e_{k_p}\} \mbox{ \ \ and \ \ } 
    D_i^{out}=\{ e_{r_1},  ... , e_{r_q} \}. 
\end{equation*}
Then of course by the Kirchhoff condition
\begin{equation*}
    \sum_{i=1}^p u^2_{k_i}(l_{k_i},t)=\sum_{j=1}^q u^2_{r_j}(0,t),
\end{equation*}
and for appropriate constants $\theta_{r_j}\leq 1$, $j=1,\ldots,q$,
\begin{equation*}
    |\partial_t u_{r_j}^2|(0,t)\leq \theta_{r_j} \sum_{i=1}^p |\partial_t u_{k_i}|(l_{k_i},t).
\end{equation*}
Taking into account multiplicity of incoming and outgoing edges, it leads to 
\begin{multline} \nonumber
       \frac{d}{dt}\left( \text{deg}_{-}(v_i)\sum_{i=1}^p\int_{e_{k_i}} |\partial_x u_{k_i}^2|dx + \text{deg}_{+}(v_i)\sum_{j=1}^q\int_{e_{r_j}} |\partial_x u_{r_j}^2|dx \right)+ \text{deg}_{+}(v_i)\sum_{j=1}^q |\partial_t u_{r_j}^2 |(l_{r_j},t)\\
       \leq \text{deg}_{-}(v_i)\sum_{i=1}^p
    |\partial_t u_{k_i}^2|(0,t).
\end{multline}
The rest of estimates follows as for honeycomb thee in Lemma \ref{lem:p_tv-est}.

\end{proof}
\section{Stitching solutions on the honeycomb tree} \label{sec:gen_NBE}

In this part we generalise the considerations from Subsection \ref{sec:posit_BC} into the case of solutions of an arbitrary sign. With no surprise, the major problem of this construction is a determination of physically justified behaviour in vertices for different sign velocities at adjacent edges. In  the whole Section \ref{sec:gen_NBE}  we restrict ourselves to honeycomb trees since they consist of exactly two kinds of vertices which additionally provide the same possible cases to consider.

\subsection{Derivation of transmission conditions}\label{sec:gen_trans}

\begin{comment}
\begin{tikzpicture}

\def\tmp{1.7};

\foreach \x in {0,1}{\path[draw=gray,dashed] (0,\tmp*\x)
node[
regular polygon,%fill=gray!50,%wypelłnienie kolorem
draw=none,
regular polygon sides=6,
draw,
inner sep=.6cm,
] (hexagon) {};}	
					
\path[draw=gray, dashed] (0.985*3/2,\tmp/2)
node[
regular polygon,%fill=gray!50,%wypelłnienie kolorem
draw=none,
regular polygon sides=6,
draw,
inner sep=.6cm,
] (hexagon) {};	
	
\draw[->,thick] (-1/2,\tmp/2) -- node[below=2pt] {$e_j$} (1/2,\tmp/2);
\draw[->,thick] (1/2,\tmp/2) -- node[left=2pt] {$e_k$} (1,\tmp);
\draw[->,thick] (1/2,\tmp/2) -- node[right=1pt] {$e_l$} (1,0);
\draw (1/2,\tmp/2) node[right=2.5pt]{$v$};			
\end{tikzpicture}
\end{comment}

In order to keep the well-posedness of the solution in the terms of the distributional formulation, see reasoning in \eqref{eq:weak} -- \eqref{eq:parts}, we  are required to control the Kirchhoff conditions \eqref{eq:Kirch}. Using the notation introduced in Subsection \ref{sec:posit_BC}, we denote by $t^{\mp}$ time shortly before/after the flow through the vertex at $t$. Denote the set of \emph{edges in which the mass enters the vertex $v_i$} at $t>0$ by $\mathcal{F}_i(t):=\mathcal{F}_i^{in}(t)\cup \mathcal{F}_i^{out}(t)$ where
\begin{eqnarray}
\mathcal{F}_i^{in}(t):=\left\{e_j\in D_i^{in}:\,\,u_j(1,t^-)\geq 0\right\} \quad\text{and} \quad \mathcal{F}_i^{out}(t):=\left\{e_j\in D_i^{out}:\,\,u_j(1,t^-)\leq 0\right\}.
\end{eqnarray}
Furthermore, we need to specify the direction of a flow through the vertex. We say that \emph{flow agrees with (is opposite to) the direction of a vertex} $D_i=(D_i^{in},D_i^{out})$ at $t>0$, for some $v_i\in V$, if
\begin{equation*}
\sum_{e_j\in \mathcal{F}_i^{in}(t)} u_j^2(1,t^-)\gtrless \sum_{e_j\in \mathcal{F}_i^{out}(t)} u_j^2(0,t^-).
\end{equation*}
%On the other hand, \emph{flow is opposite to the direction of a vertex} $D_i=(D_i^{in},D_i^{out})$ at $t>0$, for some $v_i\in V$, if
%\begin{equation*}
%\sum_{e_j\in \mathcal{F}_i^{in}} u_j^2(1,t^-)<\sum_{e_j\in \mathcal{F}_i^{out}} u_j^2(0,t^-).
%\end{equation*}
If there is an equality in the above equation, then there is no flow through the vertex $v_i$ at $t$. Let us define  $\mathcal{D}_i=\left(\mathcal{D}_i^{in},\mathcal{D}_i^{out}\right)$ a \emph{flow direction of a vertex $v_i$} which is a counterpart of vertex direction in the case of metric graph. Namely,
\begin{eqnarray}\label{eq:flow_dir}
\mathcal{D}_i^{in/out}=\left\{\begin{array}{ll}
     D_i^{in/out}&\text{for flow that agrees with the direction of a vertex},  \\
     D_i^{out/in}& \text{for flow opposite to the direction of a vertex}.
\end{array}\right.
\end{eqnarray}
We say that the \emph{flow direction is positive (negative)} in the first (second) case in \eqref{eq:flow_dir} and write respectively $\text{sgn}(\mathcal{D}_i)=1$ ($\text{sgn}(\mathcal{D}_i)=-1$). 

In the following considerations we redefine the maximal and minimal transmission solver $TS_i^z(t)$, $z=m,M$, generalising conditions presented in Subsection \ref{sec:posit_BC}. We assume
\begin{enumerate}
    \item[(i)]  Kirchhoff conditions \eqref{eq:Kirch},
    \item[(ii)] continuity conditions in vertices different then sources or sinks generalising \textit{(LC)}, namely for a.e. $t\in(0,T)$
\begin{equation}\tag{\textit{FC}}
u_j(v_i,t^+)=\left\{\begin{array}{ll} u_j(v_i,t^-)& \text{for}\,\,e_j\in \mathcal{D}_i^{in}\cap \mathcal{F}_i(t),\\
0& \text{for}\,\,e_j\in \mathcal{D}_i^{in}\setminus \mathcal{F}_i(t),
\end{array}\right.
\end{equation}

\item[(iii)] energy minimization/maximization condition, with function $\mathcal{E}_i:\Pi_{e_j\in D_i^{in}\cap D_i^{out}}U_j\rightarrow \mathbb{R}$ being ge\-ne\-ra\-li\-zation of \eqref{eq:Energy}, given by the formula
	\begin{equation*}
	\begin{array}{lcl}
	\mathcal{E}_i(u(v_i,t))&=& \displaystyle  \text{sgn}(\mathcal{D}_i) \left[ \sum_{j:\,e_j\in \mathcal{D}_i^{in}} \mathcal{E}_{ij}^+(u(v_i,t))+\sum_{j:\,e_j\in \mathcal{D}_i^{out}} \mathcal{E}_{ij}^-(u(v_i,t))\right], \\ [.5cm]
	\mathcal{E}_{ij}^{\pm}(u(v_i,t))&=& \displaystyle  \frac{u_j^3(v_i,t^{\mp})-u_j^3(v_i,t^{\pm})}{3}\\
	&-&\displaystyle  \frac{\left(u_j(v_i,t^{\mp})-u_j(v_i,t^{\pm})\right)^3}{12}\theta\left(\text{sgn}(\mathcal{D}_i)\left(u_j(v_i,t^{\mp})- u_j(v_i,t^{\pm})\right)\right),
    \end{array}
    \end{equation*}
with a domain
\begin{comment}
\begin{equation}\label{eq:U_j}
    U_j=\left\{\begin{array}{ll}
         \mathbb{R}&\text{for}\,\,e_j\in \mathcal{D}_i^{in},\\[.1cm]
         \left(\min(-u_j(v_i,t^-), \text{sgn}(\mathcal{D}_i)\infty) ,\max(-u_j(v_i,t^-), \text{sgn}(\mathcal{D}_i)\infty)\right)\cup\left\{-u_j(v_i,t^-)\right\} &\text{for}\,\,e_j\in \mathcal{D}_i^{out} \cap \mathcal{F}_i^{out},\\ [.1cm]
         \left(\min(0, \text{sgn}(\mathcal{D}_i)\infty) ,\max(0, \text{sgn}(\mathcal{D}_i)\infty)\right)\cup\left\{0\right\} &\text{for}\,\,e_j\in \mathcal{D}_i^{out} \setminus \mathcal{F}_i^{out}.
         \end{array}\right.
\end{equation}
\end{comment}
\begin{eqnarray}\nonumber
    U_j&=&\mathbb{R}\qquad \text{for}\,\,e_j\in \mathcal{D}_i^{in},\\\nonumber
    U_j&=&\left(\min(-u_j(v_i,t^-), \text{sgn}(\mathcal{D}_i)\infty) ,\max(-u_j(v_i,t^-), \text{sgn}(\mathcal{D}_i)\infty)\right)\cup\left\{-u_j(v_i,t^-)\right\},\\ \nonumber
    &&\hspace{9.5cm}\text{for}\,\,e_j\in \mathcal{D}_i^{out} \cap \mathcal{F}_i^{out},
\\\nonumber
    U_j&=&\left(\min(0, \text{sgn}(\mathcal{D}_i)\infty) ,\max(0, \text{sgn}(\mathcal{D}_i)\infty)\right)\cup\left\{0\right\} \quad\text{for}\,\,e_j\in \mathcal{D}_i^{out} \setminus \mathcal{F}_i^{out}.
\end{eqnarray}
\item[(iv)] decreasing flow with respect to edge enumeration in the case of $\mathcal{E}_i$ maximization, \textit{(DF)}. 
\end{enumerate}

It is easy to notice that the form of condition $(\textit{FC})$ assures that there is no flow within the sets $\mathcal{D}_i^{in/out}$. In the case of condition (iii) the generalisation is based on the $\mathcal{E}_i$ domain's change. The restriction of the value of solutions for $e_j\in \mathcal{D}_i^{out}$ prevents the situation that there exists an edge $e_j(0)=v_i$ (resp. $e_j(l_j)=v_i$), in which the flow direction at $e_j(0)$ (resp. $e_j(l_j)$) is opposite to the flow at the vertex $v_i$ and there is no shock at $e_j(0)$ (resp. $e_j(l_j)$).  

Based on $TS^z_i$, $z=m,M$, which satisfy conditions (i) -- (iii); we can repeat the definition of $(TS_i^z)^{\star}$, $z=m,M$ given in Definition \ref{def:mM_RS}. Finally, we are ready to present transmission conditions derived by $(TS_i^z)^{\star}$ for the honeycomb tree.\newline

{\sc Case I.} Sources and sinks
\smallskip

In analogy to non-negative case we assume that for $v_i$ being a source (a sink) we have $u_j(v_i,t)=0$, $e_j\in D^{out}_i$ ($e_j\in D^{in}_i$). \newline

{\sc Case II.} Vertices from a path graph
\smallskip

Let $v_i$ be a vertex related to the path graph such that $D_i=(\left\{e_j\right\}, \left\{e_k\right\})$. By \textit{(FC)} we have the behaviour analogous to the mono-dimensional case.
\begin{enumerate}
    \item[1.] $u_j(1,t^-)\geq 0$ and $u_k(0,t^-)\leq 0$\newline
        
    \vspace{-.3cm} \noindent Let us specify the flow through the vertex.
    \begin{itemize}
    \item[a)] $u_j^2(1,t^-)\geq u_k(0,t^-)$
    We have the flow that agrees with the direction of vertex and
    \begin{equation}\label{eq:pcase1a}
        u_j(1,t^+)=u_k(0,t^+)=u_j(1,t^-).
    \end{equation}
    \item[b)] $u_j^2(1,t^-)< u_k(0,t^-)$
    We have the flow opposite to the direction of vertex and
    \begin{equation}\label{eq:pcase1b}
        u_j(1,t^+)=u_k(0,t^+)=u_k(1,t^-).
    \end{equation}
    \end{itemize}
    
    \item[2.] $u_j(1,t^-)< 0$ and $u_k(0,t^-)> 0$\newline
        
    \vspace{-.3cm} \noindent There is no flow that directs the vertex hence
    \begin{equation*}
        u_j(1,t^+)=u_k(0,t^+)=0.
    \end{equation*}
    
    \item[3.] $u_j(1,t^-)\cdot u_k(0,t^-)\geq 0$
    \newline
        
    \vspace{-.3cm} \noindent We have either \eqref{eq:pcase1a} for $\text{sgn}(u_j(1,t^-)\geq 0$, or \eqref{eq:pcase1b} for $\text{sgn}(u_j(1,t^-)\leq 0$.
\end{enumerate}

{\sc Case III.} Vertices of the hexagonal grid of the first and second kind
\smallskip

For the illustration see Figure \ref{fig:HT1&2} with the notation changed from $e_1,e_2,e_3$ to respectively $e_j,e_k,e_l$. It is worth mentioning that considerations for vertices of the first and second kind are analogous hence we concentrate only on a vertex of a second kind, see Fig. \ref{fig:HT1&2}(ii).
    \begin{enumerate}
        \item[1.] $u_j(1,t^-)\geq 0, \qquad u_k(0,t^-) \geq 0, \qquad u_l(0,t^-) < 0$\newline
        
        \vspace{-.3cm} \noindent Firstly we need to specify the direction of flow through the vertex. 
        \begin{itemize}
            \item[a)] $u_j^2(1,t^-) \geq u_l^2(0,t^-)$\newline
        
        \vspace{-.3cm} \noindent In this case the flow agrees with the direction of a vertex (goes through the vertex to the right) and therefore values of solution after the flow should not depend on values in edges from $D_i^{out}$. Intuitively, the character of the vertex should therefore fit to the case of constant sign flows. Obviously $u_j(1,t^+)=u_j(1,t^-)$. Consider now three cases related to the choice of maximal and minimal transmission solver. 
        \begin{enumerate}
            \item[--] In the first (maximal) one, $k<l$, the total energy should go to the edge $e_k$, and zero to $e_l$. Since $u_l(0,t^-)<0$ then the flow reaches the vertex and the influence of this flow needs to be somehow balanced to maintain the proper direction of a flow. Therefore, we divide the flow from $e_j$ into two parts in such a way that 
            \begin{equation}\label{eq:case1a1}
            u_l(0,t^+)=-u_l(0,t^-) \mbox{ \ \ and \ \ } u_k(0,t^+)=\sqrt{u_j^2(1,t^-) - u_l^2(0,t^-)}.
            \end{equation}
            \item[--] The second (maximal) case, $k>l$, is when the whole energy is going to $e_l$, then 
            \begin{equation}\label{eq:case1a2}
             u_k(0,t^+)=0 \mbox{ \ \ and \ \ } u_l(0,t^+)=u_j(1,t^-).
            \end{equation}
            \item[--] The last case, related to energy minimization, is more involved. We should have \begin{equation}\label{eq:case1a3}
            u_k(0,t^+)=u_l(0,t^+)=\frac{\sqrt{2}}{2}u_j(1,t),
            \end{equation} 
            but it is valid only for $\frac{\sqrt{2}}{2} u_j(1,t^-)\geq u_l(0,t^-)$. Otherwise it does not agree with the domain $U_l$. Instead, minimum is attained at the boundary of $U_l$, hence we arrive at \eqref{eq:case1a1}.
            % \begin{equation}\label{eq:1a3.2}
            % u_l(0,t^+)=-u_l(0,t^-) \mbox{ \ \ and \ \ } u_k(0,t^+)=\sqrt{u_j^2(1,t^-)-u_l^2(0,t^-)}.
            %\end{equation}
        \end{enumerate}
        \item[b)] $u_j(1,t^-)^2 < u_l(0,t^-)^2$\newline
        
        \vspace{-.3cm} \noindent Now the flow is opposite to the direction of a vertex (goes through the vertex to the left) and therefore values of solution after the flow should not depend on values in edges from $D_i^{in}$. We put 
        \begin{equation}\label{eq:case1b}
        u_k(0,t^+)=0, \quad u_l(0,t^+)=u_k(0,t^-) \mbox{ \ and \ } u_j(1,t^+)=u_l(0,t^-),
        \end{equation}
        where the last quantity 
        is negative. It is the only possibility.
        \end{itemize}
    \item[2.] $u_j(1,t^-)\geq 0, \qquad u_k(0,t^-) < 0, \qquad u_l(0,t^-) \geq 0$\newline
        
    \vspace{-.3cm} \noindent This case is analogical to 1. due to the symmetry of the honeycomb tree. 
    
    \item[3.] $u_j(1,t^-)\leq  0, \qquad u_k(0,t^-) \geq 0, \qquad u_l(0,t^-) \geq 0$\newline
        
    \vspace{-.3cm} \noindent This case is trivial since the mass flows in the direction opposite to the vertex at all edges and the vertex becomes a kind of source. The only possible boundary constraint is
\begin{equation}\label{eq:case3}
    u_j(1,t^+)=u_k(0,t^+)=u_l(0,t^+)=0.
\end{equation}

    \item[4.] $u_j(1,t) \geq 0, \qquad u_k(0,t) \leq 0, \qquad u_l(0,t) \leq 0$\newline
    
    \vspace{-.3cm} \noindent Now the situation is more interesting since the vertex resembles a sink and again there is a need to specify the direction of a flow.
    \begin{itemize}
        \item[a)] $u_j^2(1,t^-) \leq u_k^2(0,t^-) + u_l^2(0,t^-)$\newline
        
    \vspace{-.3cm} \noindent The flow is opposite to the direction of a vertex (goes through the vertex to the left) and the shock wave appears on the edge $e_j$. Obviously $u_z(0,t^+)=u_z(0,t^-)$ for $z=k,l$ and
    \begin{equation}\label{eq:case4a}
     u_j(1,t^+) = -\sqrt{u_k^2(0,t^-) + u_l^2(0,t^-)}.
    \end{equation}
    \item[b)] $u_j(1,t^-)^2 > u_k^2(0,t^-) + u_l^2(0,t^-)$\newline
        
    \vspace{-.3cm} \noindent The flow agrees with the direction of a vertex (goes through the vertex to the right) and the shock wave appears on the edge $e_j$ and we need to choose the condition for $e_k(0)$ and $e_l(0)$ at $t^+$. Again by the energy maximization methods we have two options.
    \begin{itemize}
        \item[--] for $k<j$ we repeat condition \eqref{eq:case1a1},
        \item[--] for $k>j$
        \begin{equation}\label{eq:case4b2}
        u_k(0,t^+)=-u_k(0,t^-) \mbox{ \ and \ } u_l(0,t^+)=\sqrt{u_j(1,t^-)^2 - u_k(0,t^-)^2}.
        \end{equation}
        %\item[--]  $u_l(0,t^+)=0,$\hspace{.5cm} $ u_k(0,t^+)^2=u_j(1,t^-)^2 - u_l(0,t^-)^2$, \hspace{.5cm} for $k<l$.
    \end{itemize}
    While in the case of minimization
    \begin{itemize}
        \item[--] for $\frac{\sqrt{2}}{2}u_j(1,t^-)>\max(-u_k(1,t^-),-u_l(1,t^-))$ we have \eqref{eq:case1a3}
        %\begin{equation}
        %   u_{k}(0,t^+)=u_{l}(0,t^+)=\frac{\sqrt{2}}{2}u_j(1,t^-)
        %\end{equation}
         \item[--] for $\frac{\sqrt{2}}{2}u_j(1,t^-)>-u_k(1,t^-)$ and $\frac{\sqrt{2}}{2}u_j(1,t^-)<-u_l(0,t^-)=\frac{\sqrt{2}}{2}u_j(1,t^-)+\alpha$, $\alpha^2>\sqrt{2}u_j$ we arrive at \eqref{eq:case1a1}.
         \item[--] finally for $\frac{\sqrt{2}}{2}u_j(1,t^-)>-u_l(1,t^-)$ and $\frac{\sqrt{2}}{2}u_j(1,t^-)<-u_k(0,t^-)=\frac{\sqrt{2}}{2}u_j(1,t^-)+\alpha$, $\alpha^2>\sqrt{2}u_j$ we obtain \eqref{eq:case4b2}.
    \end{itemize}
    \end{itemize}
     
    \end{enumerate}

%In order to summarise the above consideration we gathered all the cases in the Table \eqref{eq:}. We use the following notation
%\begin{equation*}
%    u_l^{\pm}=u_z(v_i,t^{\pm}), \qquad \text{for}\,\,z=j,k,l.
%\end{equation*}

\subsection{Different sign solutions}

In this part we construct an approximation of a solution which consists piece-wise of elements from classes $\mathcal{W}^+$ and $\mathcal{W}^-$. We say that $f\in \mathcal{W}^-$ if $-f \in \mathcal{W}^+$, for $\mathcal{W}^+$ defined in \eqref{eq:W+}. Let us explain now how to stitch two mentioned types of solutions. 

Let $(U_k)_{k\in K}$ be a \emph{partition of a set $d(E)$ of metric edges} of $\mathcal{G}=(G,d)$, namely a family of closed and connected intervals such that for any $U_k$ there exits exactly one metric edge $e_{k_j}$ such that $U_k\subset e_{k_j}$,
\begin{equation*}
    \bigcup_{k\in K} U_k = d(E)\qquad \text{and}\qquad \text{int } U_k \cap \text{int } U_l= \emptyset \quad \mbox{ for } k\neq l.
\end{equation*}
Define now a class of solutions $\mathcal{W}$ such that for any fixed $\mathring{u} \in TV(\mathcal{G})$
\begin{multline}\label{eq:W}
    u \in \mathcal{W} \mbox{ \ \ iff \ \ } 
    \{  u \in B(\mathcal{G}): \mbox{there exists a partition}\,\, (U_k)_{k\in K}\,\,\mbox{of a set of metric edges}\,\,d(E) \\
    \mbox{such that either }\left.u\right|_{U_k}\in \mathcal{W}^+  \mbox{  or} \left.u\right|_{U_k}\in \mathcal{W}^- \}.
\end{multline}
\begin{prop}\label{prop:W}
Let $\mathcal{G}$ be a metric honeycomb tree. 
Then the class $\mathcal{W}$ is preserved by the flow generated by the Burgers' equation \eqref{eq:p_Burgers} and the total variation norm is controlled in time, namely
\begin{equation}\label{eq:fin_est}
    \sup_{t\in [0,T]} \int_{\mathcal{G}} |\partial_x u^2| dx + \sum_j \int_0^T (|\partial_t u^2|(0,t)
    +|\partial_t u^2|(1,t)) dt \leq C\int_{\mathcal{G}} |\partial_x \mathring{u}^2| dx +
    CT\|u\|_{\infty}^3.
\end{equation}
\end{prop}

 %We aim at showing that such a class is preserved and the total variation norm is controlled in time. 
\begin{proof} We prove the proposition by stitching the solutions from $\mathcal{W}^+$ and $\mathcal{W}^-$ in several steps. 

\smallskip

{\sc Step 1.} In order to construct the general solution we introduce auxiliary solutions related to each of $U_k$. Let $u^{(k)}$ be a solution to the Burgers' equation on $\mathcal{G}$ initiated by the initial datum
\begin{equation}
    \mathring{u}^{(k)}=\mathring{u} \chi_{U_k}.
\end{equation}
Since $\mathring{u}\in \mathcal{W}^{\pm}$, it follows that $\mathring{u}^{(k)}\in \mathcal{W}^{\pm}$ and consequently, by Proposition \ref{prop:W^+}, $u^{(k)}$ is a constant sign solution over the graph.

\smallskip 

{\sc Step 2.} Now we define the interaction between two neighbouring solutions in the interior of $e_j$. Introduce function $u^{(kl)}$, for two chosen  intervals $U_k$ and $U_l$ such that $D_k \cap D_l \ni \xi(0)$ for some $\xi(0)\in (0,l_j)$. %Moreover there exists $j$ such that $s(0) \in int \; e_j$. 
We need to determine the evolution of the contact point $\xi(t)$ starting from $\xi(0)$. Without loss of generality assume that $\min U_k<\min U_l$.
\begin{enumerate}
    \item[(i)] If $u^{(k)}<0<u^{(k)}$, then solution in the neighbourhood of $\xi(0)$ is constructed as a rarefaction wave, namely
    \begin{equation*}
        u^{(kl)}(x,t)=\left\{\begin{array}{ll}
            u^{(k)}(x,t)&\text{for}\,\,\frac{x}{t}<u_k(\xi(t),t),\\[.1cm]
            \frac{x}{t}&\text{for}\,\,u_k(\xi(t),t)<\frac{x}{t}<u_l(\xi(t),t),\\ [.1cm]
            u^{(l)}(x,t)&\text{for}\,\,u_l(\xi(t),t)<\frac{x}{t}.
        \end{array}\right.
    \end{equation*}
    \item[(ii)] If $u^{(k)}>0>u^{(k)}$, then $u^{(k)}$ and $u^{(l)}$ are stitched together by the Rankine–Hugoniot condition 
    \begin{equation*}
 \frac{d}{dt} s(t)=\frac{u^{(k)}(\xi(t),t)+
    u^{(l)}(\xi(t),t)}{2}.
\end{equation*}
In the neighbourhood of $(\xi(0),0)$ we have
    \begin{equation*}
        u^{(kl)}(x,t)=u^{(k)}(x,t) \mbox{ \ for \ }
        x < \xi(t) \mbox{ \ \ and \ \ } 
        u^{(kl)}(x,t)=u^{(l)}(x,t) \mbox{ \ for \ }
        x> \xi(t).
    \end{equation*}

\end{enumerate}
%The solution on the left hand side is negative and on the right hand side it is positive. Then the evolution of the contact point is trivial
%\begin{equation}
%    \oadd{s(t)=s(0)}.
%\end{equation}

%The second possibility is when the left hand side is positive and the right hand is negative. Then solutions $u^{(k)}$ and $u^{(l)}$ are stitched together by the Rankine–Hugoniot conditions. So 
%\begin{equation}
%    \frac{d}{dt} s(t)=\frac12\big(u^{(k)}(t,s(t))+
%    u^{(l)}(t,s(t))\big).
%\end{equation}
%In that case around point $s(0)$ we have

%\begin{equation}
%    u(t)=u^{(k)}(t,x) \mbox{ \ for \ }
%    x < s(t) \mbox{ \ \ and \ \ } 
%    u(t)=u^{(l)}(t,x) \mbox{ \ for \ }
%    x> s(t).
%\end{equation}

\smallskip 

{\sc Step 3.} Finally, we concentrate on the case of changing the sign at the vertex using the transmission solver derived in Subsection \ref{sec:gen_trans}. For given conditions in vertices at $t^-$ there exists a unique representation after the flow through the vertex, at $t^+$. %Here we use the rules chosen in the section xx. The procedure is the following. If at $t^-$ at vertex $v$ we find one of the cases pointed in Section xx. Then we are given the boundary conditions at ends of edges connected to vertex $v$.
Since the flow through the vertex $\mathcal{D}_i$ is fixed, we solve the equations at outgoing edges $\mathcal{D}_i^{out}$ knowing that at least locally near the vertex the solutions are of constant sign.

Let us become more precise about the choice of the time interval where the solution is defined. We consider the case III.1a) from Subsection \ref{sec:gen_trans}. We build the solution on edges $e_j$ and $e_l$ for some time $T_1>0$, and the transmission condition gives the boundary data for the equation at $e_k$, at least in a vicinity of the vertex. Then we solve Burgers' equation in the interior of an edge $e_k$ obtaining the solution locally in time. In general it may happen that the solution $u_k$ stays positive at the vertex just for time $T_2>0$, which can be smaller than $T_1$. Hence, the procedure of deriving the solution in the neighbourhood of the vertex is well-defined in time being the minimum of $T_1$ and $T_2$. Nevertheless, since the speed of a wave propagation and the number of vertices is finite, considered time always exists.  
%this procedure gives a local in time existence since $u$ is globally bounded, so the speed of propagation of information is bounded and the total number of vertices is finite too. 
Note that the construction of the solution bases on approximation in $\mathcal{W}$-class. It follows that transmission conditions need to be modified by a suitable approximation, with some error which is controlled. To preserve the $\mathcal{W}$ -- class the boundary term must be in $\mathcal{W}_{opp}$, and this modification is explained in the next step.

\smallskip 

{\sc Step 4.} Steps 1--3 allow for a unique definition of solution for any time since the structure of the $\mathcal{W}$-class guarantees that the solutions locally are of constant sign on edges and they are uniquely determined in vertices. 
%In conclusion we claim we constructed the solution for arbitrary time, since the structure of the $\mathcal{W}$-class guarantees that the solutions locally are of constant signs.
%In addition, since the boundary conditions at outlets are given in terms of inlets, there are, by the fixed rules, uniquely determined, thus that solutions are unique.
At the end we need to estimate $TV$-norm. Repeating the considerations from the proof of Lemma \ref{lem:p_tv-est}, we note that for each edge we find the following bound
\begin{equation}
    \frac{d}{dt} \int_{e_j} |\partial_x u_j^2|dx 
    +|\partial_t u_j^2(1,t)| \,{\rm sgn\,} (u_j(1,t)) - 
    |\partial_t u_j^2(0,t)|\,{\rm sgn  \,}(u_j(0,t)) \leq 0.
\end{equation}
Of course, the above inequality does not deliver needed information, since in general not only we fail to control the boundary terms, but also the sign of the solution at the ends of the edge. 

However, based on construction proposed for $\mathcal{W}^+$-functions, we obtain local versions of the above inequality. Introduce $\pi:e_j \to [0,1]$ a smooth function such that $supp\, \pi \subset \subset e_j$ and $\pi\equiv 1$ on the internal interval in $e_j$. Then
\begin{equation}
   2\partial_t(\pi u_j) + 2u_j \partial_x(\pi u_j)- u_j^2 \pi_x=0
\end{equation}
Then we find
%We show that 
\begin{equation}
    \frac{d}{dt} \int_{e_j} |\partial_x
    (\pi u_j^2)|dx \leq \left\| \pi_x\right\|_\infty
    \left\|u_j\right\|^3_\infty.
\end{equation}
So it gives information about the interior of edges. 

However the key element is in vertices so for each vertex we use again the localization argument. Again, in order to explain the construction of local estimation we consider a concrete case from Subsection \ref{sec:gen_trans}, namely the case III.1a). %Take the vertex of type (..) and consider the case
%\begin{equation*}
%    u_j(1,t) >0, \qquad u_k(0,t) <0, \qquad u_l(0,t) >0.
%\end{equation*}
%Somehow it follows that $u_j^2(1,t)>u_k^2(0,t)$ and
Let us remind that solution is given by $u_l^2(0,t)=u_j^2(1,t)-u_k^2(0,t)$. Before we start the estimation, let us look closer at this definition. We aim at construction of the flow in the $\mathcal{W}$ -- class, so the boundary condition is required to be in $\mathcal{W}^+_{opp}$. However the above formula does not ensure that it holds. But from (\ref{cond-dt}) we deduce that $\int_0^T |\partial_t u_l^2|(0,t) dt$ is bounded. Thus, given $\epsilon >0$ we find a new $u_l^{new}(0,t) \in
\mathcal{W}^+_{opp}$ such that $\int_0^T |\partial_t {u_l^{new}}^2|(0,t) dt \leq \int_0^T |\partial_t u_l^2|(0,t) dt$ and $\|u_l^{new}(0,\cdot)-u_l(0,\cdot)\|_{L^1(0,T)} \leq \epsilon$. This way the $\mathcal{W}$ structure of solutions is preserved, and the $TV$ - norm over $\mathcal{G}$ is controlled too.

Take $\pi$ defined around the vertex $v_i$, being $1$ over 
a sufficiently large cover of $v_i$ and supported in 
$e_j\cup e_k \cup e_j$.
Then we find
\begin{equation}
    \frac{d}{dt} \int_{e_j} |\partial_x
    (\pi u_j^2)|dx 
    +|\partial_t u^2_j|(1,t) \leq \|\pi_x\|_\infty 
    \|u\|^3_\infty.
\end{equation}   
\begin{equation}
    \frac{d}{dt} \int_{e_k} |\partial_x
    (\pi u_k^2)|dx 
    +|\partial_t u^2_k|(0,t) \leq \|\pi_x\|_\infty 
    \|u\|^3_\infty.
\end{equation}
\begin{equation}
    \frac{d}{dt} \int_{e_l} |\partial_x
    (\pi u_l^2)|dx 
     \leq |\partial_t u^2_l|(0,t)+ \|\pi_x\|_\infty 
    \|u\|^3_\infty.
\end{equation}
Since $u_l^2(0,t)=u_j^2(1,t)-u_k^2(0,t)$, we conclude that
\begin{equation}\label{cond-dt}
    |\partial_t u_l^2|(0,t)\leq |\partial _t u_j^2|(1,t) +|\partial_t u_k^2|(0,t).
\end{equation}
So summing all together we get 
\begin{equation}\label{www1}
    \frac{d}{dt} \left( \int_{e_j} |\partial_x
    (\pi u_j^2)|dx+\int_{e_l} |\partial_x
    (\pi u_l^2)|dx +\int_{e_k} |\partial_x
    (\pi u_k^2)|dx\right) \leq C\|\pi_x\|_\infty 
    \|u\|^3_\infty.
\end{equation}
Although above information is sufficient, we can have even stronger condition which controls the transmission relation. Because of form of the inequalities for $e_j$
and $e_k$, taking it twice, we improve (\ref{www1}), namely  
\begin{multline}\label{www2}
    \frac{d}{dt} \left( \int_{e_j} |\partial_x
    (\pi u_j^2)|dx+\int_{e_l} |\partial_x
    (\pi u_l^2)|dx +\int_{e_k} |\partial_x
    (\pi u_k^2)|dx\right)
    \\
    +\left(
    |\partial_t u^2_l|(0,t)+|\partial_t u^2_k|(0,t)+|\partial_t u^2_j|(1,t)\right)
    \leq C\|\pi_x\|_\infty 
    \|u\|^3_\infty.
\end{multline}

In the general case, the signs at the vertex may be different, we get the better info with boundary term for the case when the flow at the edge comes into the vertex. So such inequality we make double, then we obtain (\ref{www2}) for the general case. Note that there is only one case where there is no incoming flow, but then all boundary terms are just zero, so the time derivatives vanish too.

Finally, repeating the steps from Lemma \ref{lem:p_tv-est} in the general case, we get \eqref{eq:fin_est}.
\end{proof}

%We claim 
%\begin{equation}\label{W-tv-est}
%    \sup_t \|u(t)\|_{TV(\mathcal{G})} 
%    \leq \sum_k \sup_t \|u^{(k)}\|_{TV(\mathcal{G})}
%    \leq 2^{\kappa_{\mathcal{G}}}
%    \|u_0\|_{TV(\mathcal{G})}.
%\end{equation}

%Looking at Step 2 we see that the norm can only decrease in such cases. The Step 3 has the same feature, here we shall note that creation of shocks 
%decrease the norm, so we conclude (\ref{W-tv-est}).

% \subsection{Derivation of transmission conditions}\label{sec:gen_trans}
\subsection{Existence of solution}\label{sec:gen_exist}

In the last part of this section we show the following existence result, that goes in line with Definition \ref{def:conserv}.

\medskip 
\begin{thm}\label{thm:main}
Let $\mathring{u} \in TV(\mathcal{G})$. There exists a weak solution to the Burgers' equation on graph $\mathcal{G}$ such that
\begin{equation*}
    u^2 \in L_\infty(0,T;TV(\mathcal{G})).
\end{equation*}
%it fulfils the distributional meaning and the transition conditions at vertices  holds a.a. in time.
\end{thm}

\medskip 

\begin{proof}

For given $\mathring{u} \in TV(\mathcal{G})$, let us proceed in the following steps.

{\sc Step 1.} Firstly, we approximate the initial condition. For given $\epsilon > 0$,
one finds $\mathring{u}_\epsilon=(\mathring{u}_\epsilon)_++(\mathring{u}_\epsilon)_-$ such that $(\mathring{u}_\epsilon)_+ \in \mathcal{W}^+$, 
$(\mathring{u}_\epsilon)_- \in \mathcal{W}^-$ and 
\begin{equation*}
    \|\mathring{u}-\mathring{u}_\epsilon\|_{L^1(\mathcal{G})} < \epsilon \mbox{  \ \ and \ \ } 
    \|\mathring{u}_\epsilon^2\|_{TV(\mathcal{G})}\leq \|\mathring{u}^2\|_{TV(\mathcal{G})}.
\end{equation*}
We solve the equation starting from $\mathring{u}_\epsilon$ in the class $\mathcal{W}$ according to the steps presented in Proposition \ref{prop:W}. Then the uniform bound
\begin{equation*}
    \sup_{t\in (0,T)} \|u_\epsilon^2(t)\|_{TV(\mathcal{G})} \leq C,
%\end{equation*}
\mbox{ \ \ can be found, as well as \  \ }
%\begin{equation*}
    \sup_{t\in (0,T)} \|\partial_t u_\epsilon(t)\|_{\mathcal{M(G)}} \leq C.
\end{equation*}
{\sc Step 2.}
Using Lions-Aubin lemma we find a subsequence such that
\begin{equation*}
    u_\epsilon \to u^* \in L^p(\mathcal{G} \times (0,T)) \mbox{ \ \ for  any \ } p<\infty,
\end{equation*}
hence $u^*$ is a weak solution. Weak limits guarantee that 
\begin{equation}
    u\in L^\infty(\mathcal{G}\times (0,T)) 
    \mbox{ \ and \ } u\in L^\infty(0,T;TV(\mathcal{G})).
\end{equation}

{\sc Step 3.} The boundary conditions follow from the information carried by (\ref{eq:fin_est}), while in the limit this condition can be found only as measure. The compactness ensures us that the approximating sequence goes strongly at the boundary point-wisely since then $u_\epsilon \to u$ in $L^p(0,T)$ in the vertices in time.

{\sc Step 4.} As the last step let us comment on the uniqueness. The above properties of solutions to the Burgers' equation fulfil the conditions for the classical mono-dimensional case. We obtain an entropy solution as a bounded distributional solution with the bound (\ref{eq:bound}). 

We claim that the solution is unique. Unfortunately, in order to restate the proof from Evans textbook, see \cite{Evans}, the method of characteristics on metric graphs for the transport type equation with smooth coefficients is needed. To our best knowledge still there is no such result in the literature.
It will be the subject of our further investigations, hence at this moment we state the uniqueness only as a conjecture.

\end{proof}

\section{Conclusions}\label{sec:concl} 

At the end of this paper we return to our question from the introduction in order to  deliver an interpretation of Burgers' equation in the spirit of wave interference. 

Let us start again with mono-dimensional equation, namely \eqref{eq:BE} with $D=\mathbb{R}$. 
%We start with questions concerning  interaction of waves. 
 The physical interpretation of the classical theory of Burgers' equation  delivers a poor picture of the possibility to describe more complex phenomenon than the motion of one wave. Let us explain the issue. Think about the following initial configuration on the line.
\begin{equation}\label{eq:init_cond}
    u|_{t=0}=\chi_{[-4,-3]} - \chi_{[3,4]}.
\end{equation}
For simplicity consider distributional solutions being a shift with a speed determined by the Rankine-Hugoniot condition. It means that solution at least for small time is given by
\begin{equation*}
    u(x,t)=\chi_{[-4+\frac12 t,-3+\frac12 t]}(x) - \chi_{[3-\frac12 t, 4 -\frac12 t]}(x).
\end{equation*}
%\oadd{Taking the entropy solutions obtained in the paper we would get more complex formulas, however with the same qualitative aspects. }
Furthermore, %the theory says that 
for a long time the solution disappears. In the case of different velocities of waves, the stronger one overtakes the smaller one which is the consequence of the weak formulation and the regime of the Rankine-Hugoniot conditions. From mathematical perspective we are not allowed to obtain more interesting configurations, since the regime of distributional solutions makes the choice of possible evolution very restrictive. However looking at the solution in from physical perspective % different way, using the method of characteristics 
the behaviour described above is not so obvious. %Of course the regime of distributional solutions makes the choice of possible evolution very restrictive. 
We can expect for example passing a smaller waves, at least partially, through a larger one. 
%One can say that system () is not able to capture such effects, simple mathematics does not allow for it.

To overcome this obstacle we put our attention on   the choice of the domain $D$. We follow the idea, that equation shall be simple, but the area of action shall give possibility to consider different branches of solutions. By definition mono-dimensional object with many possible paths is just a graph. Therefore we rewrite the system into
\begin{equation*}
    \partial_t u +u\partial_x u=0 \mbox{ on } \mathcal{G} \times [0,T), \qquad u|_{t=0}= \mathring{u} \mbox{ \ at \ } \mathcal{G},
\end{equation*}
where $\mathcal{G}$ is the metric graph. This way we shall be able to obtain a rich structure of solutions even for initial data like \eqref{eq:init_cond}. Let us look at the following example.

\begin{exam}
Let $\mathcal{G}=(G,d)$ be the following metric tree $V=\left\{v_1,v_2\right\}$, $E=\left\{e_1,\ldots,e_4\right\}$,
\begin{equation*}
\mathcal{L}(e_i)=\left\{\begin{array}{cc}
     9& \text{for}\,\,i=1,4\\
     2&  \text{for}\,\,i=2,3
\end{array}\right.,\quad \phi=    \left[\begin{array}{cccc}
     1&-1&-1&0  \\
     0&1&1&-1 
\end{array}\right],\quad d(e_i)=\left\{\begin{array}{ll}
     [0,9]& \text{for}\,\,i=1,4\\\relax
     [0,2]& \text{for}\,\,i=2,3
\end{array}\right..
\end{equation*}
Note that $\mathcal{G}$ can be interpreted as interval $[-10,10]$ splited at $(-1,0)$ into two and joined again at $(1,0)$. 

\bigskip

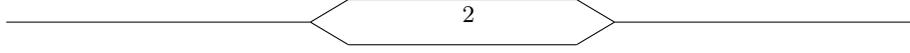
\begin{center}
\begin{tikzpicture}
\draw (-2,0) -- (-1.5,0.3);
\draw (-1.5,0.3) -- (1.5,0.3);
\draw (1.5,0.3) -- (2,0);
\draw (2,0) -- (6,0);
\draw (-6,0) -- (-2,0);
\draw (-2,0) -- (-1.5,-0.3);
\draw (-1.5,-0.3) -- (1.5,-0.3);
\draw (1.5,-0.3) -- (2,0);
\put(-120,5){$e_1$};
\put(120,5){$e_4$};
\put(0,15){$e_2$};
\put(0,-25){$e_3$};
\put(-70,3){$v_1$};
\put(60,3){$v_2$};
\end{tikzpicture}
\end{center}

\bigskip

%\noindent
%Edges:
%\begin{equation}
%    I^-=(\infty,-1), \qquad I^\uparrow = (-1,1), 
%    \qquad I^\downarrow=(-1,1), \qquad I^+=(1,\infty).
%\end{equation}
%Vertices:
%\begin{equation}
%    A=(-1,0) \mbox{ \ \ and \ \ } B=(1,0).
%\end{equation}
%Note that the picture is just an illustration. Indeed the {graph is mono-dimensional}, we have two paths between $A$ and $B$. Omitting here the details of mathematical formulation of a problem we give an example. 
\noindent We consider Burgers' equation on $\mathcal{G}$ with the following initial condition
\begin{equation}\label{eq:example_ic}
    u_1|_{t=0}=\chi_{[6,7]}, \qquad
     u_4|_{t=0}=-\chi_{[2,3]}, \qquad
     u_2|_{t=0}=u_3|_{t=0}=0.
\end{equation}
Note that condition \eqref{eq:example_ic} for network is an analogue of condition \eqref{eq:init_cond} for a straight line and at time $t=0$ we can illustrate it in the following way

\bigskip

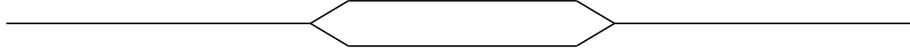
\begin{center}
\begin{tikzpicture}[auto,node distance=1.5cm,semithick]
\tikzstyle{every state}=[fill=none, draw=none, text=black]
\draw (-2,0) -- (-1.5,0.3);
\draw (-1.5,0.3) -- (1.5,0.3);
\draw (1.5,0.3) -- (2,0);
\draw (2,0) -- (6,0);
\draw (-6,0) -- (-2,0);
\draw (-2,0) -- (-1.5,-0.3);
\draw (-1.5,-0.3) -- (1.5,-0.3);
\draw (1.5,-0.3) -- (2,0);
%\put(-120,5){$I^-$};
%\put(120,5){$I^+$};
%\put(0,15){$I^\uparrow$};
%\put(0,-25){$I^\downarrow$};
\put(-70,3){$v_1$};
\put(60,3){$v_2$};
\put(-140,7){\vector(1,0){30}};
\put(140,7){\vector(-1,0){30}}
\end{tikzpicture}
\end{center}

\bigskip
\end{exam}

 To avoid problems with definitions and argumentation, we just present very schematic behaviour of the proposed system. We assume that the waves are 
non-physical, namely they are simple 
shifts in time of characteristic functions of kind $\chi_{[\frac12 t, 1+\frac 12 t]}(x)$. One can obviously consider examples base on the solutions being combination of rarefaction and shock waves, like in formula \cite[Sec.3.4.1b~Exam.3]{Evans}, but then
we loose the schematic form of the considerations. The character of dynamics will be determined by the rules in vertices, describing the partition of the solutions onto different paths. Consider three situations:

\smallskip

{\sc Case I. } In vertex $v_1$ the wave from edge $e_1$ goes on $e_2$, and in vertex $v_2$ the wave from $e_4$ goes on $e_3$. So at $t=t_1$ suitably chosen we have 
\bigskip
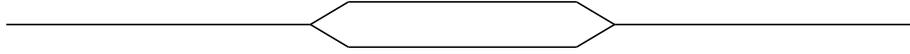
\begin{center}
\begin{tikzpicture}[auto,node distance=1.5cm,semithick]
\tikzstyle{every state}=[fill=none, draw=none, text=black]
\draw (-2,0) -- (-1.5,0.3);
\draw (-1.5,0.3) -- (1.5,0.3);
\draw (1.5,0.3) -- (2,0);
\draw (2,0) -- (6,0);
\draw (-6,0) -- (-2,0);
\draw (-2,0) -- (-1.5,-0.3);
\draw (-1.5,-0.3) -- (1.5,-0.3);
\draw (1.5,-0.3) -- (2,0);
%\put(-120,5){$I^-$};
%\put(120,5){$I^+$};
%\put(0,15){$I^\uparrow$};
%\put(0,-25){$I^\downarrow$};
\put(-70,3){$v_1$};
\put(60,3){$v_2$};
\put(-35,15){\vector(1,0){30}};
\put(35,-4){\vector(-1,0){30}}
\end{tikzpicture}
\end{center}

\bigskip

\noindent
Then waves pass through without direct interaction, so the energy is not lost. For large time 
we obtain the solution of the form
\begin{equation}
    u(x,t)= -\chi_{[3-\frac12 t,4-\frac12 t]}+
    \chi_{[-4+\frac12 t,-3+\frac12 t]},
\end{equation}
so there is no interaction of waves. Such a result is not possible in description by the classical Burgers equation.
 
\smallskip

{\sc Case II.}
In vertex $v_1$ the wave divides into two equal parts (in the sense of energy), and the same happens for vertex $v_2$. For $t=t_1$ we have 
\bigskip
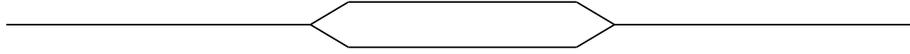
\begin{center}
\begin{tikzpicture}[auto,node distance=1.5cm,semithick]
\tikzstyle{every state}=[fill=none, draw=none, text=black]
\draw (-2,0) -- (-1.5,0.3);
\draw (-1.5,0.3) -- (1.5,0.3);
\draw (1.5,0.3) -- (2,0);
\draw (2,0) -- (6,0);
\draw (-6,0) -- (-2,0);
\draw (-2,0) -- (-1.5,-0.3);
\draw (-1.5,-0.3) -- (1.5,-0.3);
\draw (1.5,-0.3) -- (2,0);
%\put(-120,5){$I^-$};
%\put(120,5){$I^+$};
%\put(0,15){$I^\uparrow$};
%\put(0,-25){$I^\downarrow$};
\put(-70,3){$v_1$};
\put(60,3){$v_2$};
\put(-35,15){\vector(1,0){20}};
\put(35,-4){\vector(-1,0){20}};
\put(-35,-4){\vector(1,0){20}};
\put(35,15){\vector(-1,0){20}}
\end{tikzpicture}
\end{center}

\bigskip

Now the waves meet on both $e_2$ and $e_3$ and since they are anti-symmetric, they annihilate. Thus for large time 
\begin{equation}
    u(x,t)= 0.
\end{equation}
This case covers the classical result of Burgers' equation, like 
without a graph.
 
\smallskip

{\sc Case III.}
In the vertex $v_1$ the wave divides into two equal parts, but in the vertex $v_2$ the wave from $e_4$ goes on the edge $e_3$. It means that the upper part of the wave goes on $e_4$, but on lower edge $e_3$ we have a shock of
two waves. For $t=t_1$
\bigskip
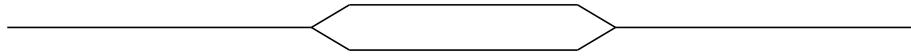
\begin{center}
\begin{tikzpicture}[auto,node distance=1.5cm,semithick]
\tikzstyle{every state}=[fill=none, draw=none, text=black]
\draw (-2,0) -- (-1.5,0.3);
\draw (-1.5,0.3) -- (1.5,0.3);
\draw (1.5,0.3) -- (2,0);
\draw (2,0) -- (6,0);
\draw (-6,0) -- (-2,0);
\draw (-2,0) -- (-1.5,-0.3);
\draw (-1.5,-0.3) -- (1.5,-0.3);
\draw (1.5,-0.3) -- (2,0);
%\put(-120,5){$I^-$};
%\put(120,5){$I^+$};
%\put(0,15){$I^\uparrow$};
%\put(0,-25){$I^\downarrow$};
\put(-70,3){$v_1$};
\put(60,3){$v_2$};
\put(-35,15){\vector(1,0){20}};
\put(35,-4){\vector(-1,0){30}};
\put(-35,-4){\vector(1,0){20}};
\end{tikzpicture}
\end{center}

\bigskip

\noindent
Since the one coming from the right side is larger, the smaller one is overtaken and the wave flows on the edge $e_1$. Hence up to a small modification of time related to Rankine-Hugoniot conditions, for large time we have 
\begin{equation*}
    u(x,t)= - \chi_{[3-\frac12 t,4-\frac12 t]}+
   \frac{\sqrt{2}}{2} \chi_{[-4+\frac12 t,-3+\frac12 t]}
\end{equation*}
This case is the most interesting since we obtain a practical interference. One part is dumped while the second one is preserved in its magnitude.
%\end{exam}

\subsubsection*{Acknowledgements} The authors have been partly supported by National Science Centre grant 2018/29/B/ST1/00339 (Opus). Additionally, the research of AP was partially supported by National Science Centre grant\linebreak 2017/25/N/ST1/00787 (Preludium).

\end{document}